\documentclass[twoside]{amsart}

\numberwithin{equation}{section}
\usepackage{amsmath,amssymb,amsfonts,amsthm,latexsym}
\usepackage{graphicx,color,enumerate}
\usepackage{hyperref}
\usepackage[margin=1.4in]{geometry}

\newtheorem{theorem}{Theorem}[section]
\newtheorem{lemma}[theorem]{Lemma}
\newtheorem{proposition}[theorem]{Proposition}
\newtheorem{corollary}[theorem]{Corollary}

\theoremstyle{definition}
\newtheorem{definition}[theorem]{Definition}
\newtheorem{example}[theorem]{Example}
\theoremstyle{remark}

\numberwithin{equation}{section}

\newcommand{\C}{\mathbb{C}}

\newcommand{\R}{\mathbb{R}}
\newcommand{\T}{\mathbb{T}}

\newcommand{\Z}{\mathbb{Z}}

\newcommand{\lcm}{\operatorname{lcm}}

\newcommand{\diag}{\operatorname{diag}}
\newcommand{\codim}{\operatorname{codim}}
\newcommand{\co}{\colon\thinspace}

\newcommand{\bs}{\boldsymbol}

\newcommand{\Imaginary}{\operatorname{Im}}

\begin{document}

\title{Constructing symplectomorphisms between symplectic torus quotients}

\author[H.-C.~Herbig]{Hans-Christian Herbig}
\address{Departamento de Matem\'{a}tica Aplicada,
Av. Athos da Silveira Ramos 149, Centro de Tecnologia - Bloco C, CEP: 21941-909 - Rio de Janeiro, Brazil}
\email{herbighc@gmail.com}

\author[E.~Lawler]{Ethan Lawler}
\address{Department of Mathematics \& Statistics,
Dalhousie University,
6316 Coburg Road,
PO BOX 15000,
Halifax, Nova Scotia,
Canada B3H 4R2}
\email{lawlerem@gmail.com}

\author[C.~Seaton]{Christopher Seaton}
\address{Department of Mathematics and Computer Science,
Rhodes College, 2000 N. Parkway, Memphis, TN 38112, USA}
\email{seatonc@rhodes.edu}

\thanks{H.-C.H. was supported by CNPq through the \emph{Plataforma Integrada Carlos Chagas},
E.L. was supported by a Rhodes College Research Fellowship,
and C.S. was supported by the E.C.~Ellett Professorship in Mathematics.}

\begin{abstract}
We identify a family of torus representations such that the corresponding singular symplectic quotients at
the $0$-level of the moment map are graded regularly symplectomorphic to symplectic quotients associated
to representations of the circle. For a subfamily of these torus representations, we give an explicit
description of each symplectic quotient as a Poisson differential space with global
chart as well as a complete classification of the graded regular diffeomorphism and symplectomorphism
classes. Finally, we give explicit examples to indicate that symplectic quotients in this class may
have graded isomorphic algebras of real regular functions and graded Poisson isomorphic complex symplectic quotients
yet not be graded regularly diffeomorphic nor graded regularly symplectomorphic.
\end{abstract}

\maketitle
\tableofcontents


\section{Introduction}
\label{sec:Intro}

Let $G$ be a compact Lie group and $G\to \mathrm{U}(V)$ a finite dimensional unitary representation of $G$. Here $\mathrm{U}(V)$ stands for the unitary group of $V$, i.e.
the group of automorphisms preserving the hermitian inner product $\langle\cdot,\cdot\rangle$. To describe the orbit space $V/G$, i.e. the space of $G$-orbits in $V$, invariant theory is employed as follows. There exists a system of fundamental real homogeneous polynomial invariants $\phi_1,\phi_2,\dots,\phi_m$; we refer to the system $\phi_1,\phi_2,\dots,\phi_m$ as a \emph{Hilbert basis}. This means that any real invariant
polynomial $f\in \mathbb R[V]^G$ can be written as a polynomial in the $\phi$'s, i.e. there exists a polynomial $g\in\mathbb R[x_1,x_2,\dots,x_m]$ such that $f=g(\phi_1,\phi_2,\dots,\phi_m)$.
More generally, by a theorem of G. W. Schwarz
\cite{SmoothInv}, for any smooth function $f\in \mathcal C^\infty(V)^G$ there exists $g\in\mathcal C^\infty(\mathbb R^m)$ such that $f=g(\phi_1,\phi_2,\dots,\phi_m)$. The vector-valued map $\phi=(\phi_1,\phi_2,\dots,\phi_m)$ gives rise to an embedding $\overline{\phi}$ of $V/G$ into euclidean space  $\mathbb R^m$, which is called the \emph{Hilbert embedding}.
We denote its image by $X:=\phi(V)$.  It turns out that
$\overline{\phi}$ is actually a \emph{diffeomorphism} onto $X$, i.e. the pullback $\overline{\phi}^\ast$
via $\overline{\phi}$ induces an isomorphism of algebras
$\mathcal C^\infty(X):=\{g \co X\to \mathbb R\mid \exists G\in\mathcal C^\infty(\mathbb R^m) \colon g=G_{|X}\}$
and $\mathcal C^\infty(V/G):=\mathcal C^\infty(V)^G$. Moreover, the restriction of
$\overline{\phi}^\ast$ to the subalgebra $\mathbb R[X]:=\{g\co X\to \mathbb R\mid \exists G\in
\mathbb R[x_1,x_2,\dots,x_m]  \colon g=G_{|X}\}$ isomorphically to $\mathbb R[V/G]:=\mathbb R[V]^G$ preserving the grading. Here we use the natural grading $\deg({x_i}):=\deg({\phi_i})$.
We say that $\overline{\phi}^\ast$ is a \emph{graded regular diffeomorphism}. The algebra $\mathbb R[X]$ can be understood as the quotient of $\mathbb R[x_1,x_2,\dots,x_m]$ by the kernel of the restriction map, which we refer to as the ideal of \emph{off-shell relations}. Its generators are assumed to be homogeneous in the natural grading. The real variety underlying $\mathbb R[X]$ is the Zariski closure $\overline{X}$ of $X$ inside $\mathbb R^m$. The space $X$ itself is not a real variety but a \emph{semialgebaic set}. How the inequalities cutting out $X$ inside $\overline{X}$ are obtained has been explained in \cite{ProcesiSchwarz}.

The hermitian vector space $V$ is equipped with the symplectic form $\omega=\mathrm{Im}\langle\cdot,\cdot\rangle$ obtained be taking the imaginary part of hermitian inner product. Moreover, the action of $G$ on
$V$ is Hamiltonian and admits a unique homogeneous quadratic moment map $J\co V\to\mathfrak{g}^\ast$
where $\mathfrak{g}^\ast$ denotes the dual of the Lie algebra $\mathfrak{g}$ of $G$.
The zero fibre $Z:=J^{-1}(0)$ of $J$ is referred to as the \emph{shell}. It is a real subvariety of $V$ with a conical singularity at the origin. Due to the $G$-equivariance of  $J$ the group $G$ acts on $Z$. The space $M_0:=Z/G$ of $G$-orbits in $Z$ is called the (linear) \emph{symplectic quotient}. By the work Sjamaar and Lerman \cite{SjamaarLerman} the smooth
structure $\mathcal C^\infty(M_0)$ is given by the quotient $\mathcal C^\infty(V)^G/\mathcal I_Z^G$ where $\mathcal I_Z^G$ is the invariant part of the vanishing ideal $\mathcal I_Z:=\{f\in \mathcal C^\infty(V)\mid f_{|Z}=0\}$. Note that $\mathcal C^\infty(M_0)$ is in a canonical way a Poisson algebra containing the Poisson subalgebra $\mathbb R[M_0]:=\mathbb R[V]^G/I_Z^G$, where $I_Z^G:=\mathcal I_Z\cap \mathbb R[V]^G$.
The image $Y:=\phi(Z)$ of $Z$ under the Hilbert map is a semialgebraic subset of $X$.
Its Zariski closure $\overline{Y}$ is described by the generators of the kernel in $ \mathbb R[x_1,x_2,\dots,x_m]$ of the algebra morphism $x_i\mapsto {\phi_i}_{|Z}\in \mathcal C^\infty(M_0)$. We refer to it as the ideal of \emph{on-shell} relations.
The inequalities that cut out $Y$ from $\overline{Y}$ are the same as those cutting out $X$ from $\overline{X}$.

Let us now assume that we have two symplectic quotients $M_0$ and $M_0^\prime$ constructed from
the representations $G\to \mathrm U(V)$ and $G^\prime\to \mathrm U(V^\prime)$, respectively. By a \emph{symplectomorphism}
between $M_0$ and $M_0^\prime$ we mean a homeomorphism $F\co M_0\to M_0^\prime$ such that the pullback $F^\ast$
is an isomorphism of Poisson algebras $F^\ast\co\mathcal C^\infty(M_0^\prime)\to \mathcal C^\infty(M_0)$.
We say that $F$ is \emph{regular} if $F^\ast(\mathbb R[M_0^\prime])\subseteq\mathbb R[M_0]$. A regular symplectomorphism is called \emph{graded regular} if the map $(F^\ast)_{|\mathbb R[M_0^\prime]}\co \mathbb R[M_0^\prime]\to\mathbb R[M_0]$ preserves the grading. By the Lifting Theorem  of \cite{FarHerSea}, an isomorphism
$f\co \mathbb R[M_0^\prime]\to\mathbb R[M_0]$ of Poisson algebras gives rise to a unique symplectomorphism if it compatible with the inequalities.

 When $G = \T^\ell$ is a torus, a representation $V$ of complex dimension $n$ can be described in terms
of a weight matrix $A \in \Z^{\ell\times n}$; we use $M_0(A)$ to denote the symplectic quotient associated
to the representation with weight matrix $A$.
In \cite[Theorem 7]{FarHerSea}, it is demonstrated that for a weight matrix of the form
$A = [D|C]$ where $D$ is an $\ell\times\ell$ diagonal matrix with strictly negative entries on the
diagonal and $C$ is an $\ell\times 1$ matrix with strictly positive entries, the corresponding symplectic
quotient by $\T^\ell$ is graded regularly symplectomorphic to the symplectic orbifold $\C/\Z_\eta$ where
$\eta = \eta(A)$ is a quantity determined by the entries of $A$; see Definition~\ref{def:TypeAlpha}.
However, based on the explicit description of the ring
$\R[\C]^{\Z_\eta}$ of real regular functions on the orbifold $\C/\Z_\eta$ given in the proof of \cite[Theorem 7]{FarHerSea},
it is easy to see that $\R[\C]^{\Z_{\eta_1}}$ and $\R[\C]^{\Z_{\eta_2}}$ are isomorphic as algebras over $\R$
if and only if $\eta_1 = \eta_2$. Hence, an immediate corollary of \cite[Theorem 7]{FarHerSea} is the following.

\begin{corollary}
\label{cor:SIGMASympCrit}
For $i = 1,2$, let $A_i = [D_i|C_i]$ where each $D_i$ is an $\ell_i\times\ell_i$ diagonal matrix with strictly
negative entries on the diagonal and each $C_i$ is an $\ell_i\times 1$ matrix with strictly positive entries.
Then the symplectic quotients $M_0(A_1)$ and $M_0(A_2)$ are regularly diffeomorphic if and only if
$\eta(A_1) = \eta(A_2)$, in which case they are graded regularly symplectomorphic.
\end{corollary}

More recently, it was shown in \cite[Theorem 1.1]{HerbigSchwarzSeaton} that for general symplectic quotients,
symplectomorphisms with symplectic orbifolds are rare, even if the graded regular
requirements are dropped; see also \cite{HerbigSeaton2}. Hence, one cannot use
isomorphisms with quotients by finite groups to approach a more general classification
of higher-dimensional symplectic quotients by tori.

In this paper, we give a generalization of Corollary~\ref{cor:SIGMASympCrit} as a step towards a general
classification of linear symplectic quotients by tori into (graded) regular symplectomorphism classes. While
Corollary~\ref{cor:SIGMASympCrit} addresses a class of symplectic quotients by tori that can be reduced
to quotients by finite groups, we consider here a class of symplectic quotients by tori that are graded
regularly symplectomorphic to symplectic quotients by the circle $\T^1$. To state our main result,
we say that a weight matrix
$A \in\Z^{\ell\times(\ell+k)}$ is \emph{Type II$_k$} if it can be expressed in the form
$A = [D, c_1\bs{n},\ldots, c_k\bs{n}]$ with $D$ a diagonal matrix with strictly negative diagonal entries,
$\bs{n}$ a column matrix with strictly positive entries, and each $c_r \geq 1$. Our main result is that the
symplectic associated to a Type II$_k$ matrix of any size is graded regularly symplectomorphic to a symplectic
quotient by $\T^1$. Specifically, we have the following; see Definition~\ref{def:TypeAlpha} for the definitions
of $\alpha$ and $\beta$.

\begin{theorem}
\label{thrm:MainSymplectomorph}
Let $A\in\Z^{\ell\times (\ell+k)}$ be the Type II$_k$ matrix of a faithful $\T^\ell$-representation
$V$ of dimension $n = \ell+k$. Then the symplectic quotient $M_0(A)$ is graded regularly symplectomorphic to
the $\T^1$-symplectic quotient $M_0(B)$ where
$B = \big(-\alpha(A), c_1\beta(A),\ldots, c_k\beta(A)\big) \in\Z^{1\times(k+1)}$.
\end{theorem}

Theorem~\ref{thrm:MainSymplectomorph} can be thought of as a dimension reduction formula, allowing one to
describe symplectic quotients by $\T^\ell$ associated to Type II$_k$ weight matrices in terms of much simpler quotients
by $\T^1$. In particular, it extends results concerning $\T^1$-symplectic quotients to this family of quotients
by tori, e.g. the Hilbert series computations of \cite{HerbigSeatonHilbSympCirc} or the representability results
of \cite{WattsSympQuotCircle}. The graded regular symplectomorphism given by the theorem preserves several structures,
and hence can be thought of as a symplectomorphism of symplectic stratified spaces, a graded isomorphism of the
corresponding real algebraic varieties, etc., and it induces a graded Poisson isomorphism of the corresponding
complex symplectic quotients, the complexifications treated as complex algebraic varieties with symplectic singularities;
see \cite{HerbigSchwarzSeaton2}.

The proof of Theorem~\ref{thrm:MainSymplectomorph} is given in Section~\ref{sec:MainGenProof} by indicating a Seshadri
section for the action of the torus on the zero fiber of the moment map after complexifying; see
\cite[Corollary, page 169]{PopovSeshadri} and \cite[Theorem 3.14]{PopovVinberg}. The first proof we obtained
of Theorem~\ref{thrm:MainSymplectomorph}, however, was constructive for a smaller class of weight
matrices, so-called \emph{Type I$_k$} (see Definition~\ref{def:TypeAlpha}), and used explicit descriptions of
the corresponding symplectic quotients and algebras of real regular functions. Because this description has proven
useful and may be of independent interest, we give this description and outline the constructive approach in
Section~\ref{sec:ConstrucProof}.

In the case of symplectic quotients of (real) dimension $2$ considered in Corollary~\ref{cor:SIGMASympCrit}
(corresponding to Type I$_1$ weight matrices),
the graded regular symplectomorphism class of $M_0(A)$ depends only on the constant $\eta(A)$, which is given by the
sum $\alpha(A) + \beta(A)$ (see Definition~\ref{def:TypeAlpha}). In the case of Type I$_k$ weight matrices with $k > 1$,
this is no longer the case; we show in Section~\ref{sec:Classification} that the graded regular symplectomorphism
classes of Type I$_k$ symplectic quotients are classified by $k$, $\alpha(A)$, and $\beta(A)$. For Type II$_k$
weight matrices, though the graded regular symplectomorphism class of $M_0(A)$ is certainly not determined by $k$
and $\eta(A)$, the situation is more subtle, and such a classification would require very different techniques.
In Section~\ref{sec:HilbCounterex}, we indicate this with examples of symplectic quotients associated to
Type II$_k$ weight matrices that fail to be graded regularly symplectomorphic, though the corresponding complex algebraic varieties are graded Poisson isomorphic, and hence the Hilbert series of real regular functions coincide.

\section*{Acknowledgements}

This is a pre-print of an article published in Beitr\"{a}ge zur Algebra und Geometrie/Contributions to Algebra and Geometry.
The final authenticated version is available online at: \newline\url{https://doi.org/10.1007/s13366-020-00486-8} .

This paper developed from EL's senior seminar project in the Rhodes College Department of Mathematics
and Computer Science, and the
authors gratefully acknowledge the support of the department and college for these activities.
C.S. would like to thank
the Instituto de Matem\'{a}tica Pura e Aplicada (IMPA) for hospitality during work contained here.
H.-C.H. was supported by CNPq through the \emph{Plataforma Integrada Carlos Chagas},
E.L. was supported by a Rhodes College Research Fellowship,
and C.S. was supported by the E.C.~Ellett Professorship in Mathematics.


\section{Background on torus representations}
\label{sec:Background}

In this section, we give a brief overview of the structures associated to (real linear) symplectic quotients
by tori, specializing the constructions described in the Introduction.
We refer the reader to \cite{FarHerSea,HerbigIyengarPflaum} for more details.

Let $G = \T^\ell$ and let $V$ be a unitary $G$-module with $\dim_\C V = n$. Choosing a basis with respect to
which the action of $G$ is diagonal and letting $\bs{z} = (z_1, \ldots, z_n) \in\C^n$ denote coordinates for
$V$ with respect to this basis, the action of $G$ is given by
\[
    \bs{t}\bs{z}
    := \big( t_1^{a_{11}} t_2^{a_{21}}\cdots t_\ell^{a_{\ell 1}} z_1,
            t_1^{a_{12}} t_2^{a_{22}}\cdots t_\ell^{a_{\ell 2}} z_2,
            \ldots,
            t_1^{a_{1n}} t_2^{a_{2n}}\cdots t_\ell^{a_{\ell n}} z_n \big)
\]
where $\bs{t} = (t_1,t_2,\ldots,t_\ell)\in G$ and $A = (a_{ij})\in\Z^{\ell\times n}$ is the \emph{weight matrix} of
the representation. Given a weight matrix $A\in\Z^{\ell\times n}$, we let $V_A$ denote the
$n$-dimensional representation of $\T^\ell$ with weight matrix $A$ along with the corresponding basis for $V_A$.
We let $\langle \cdot,\cdot\rangle$ denote the standard hermitian scalar product on $V_A$ corresponding to this
basis.

Letting $\bs{a}_j$ denote the $j$th column of $A$ so that $A = (\bs{a}_1,\ldots, \bs{a}_n)$,
it will be convenient to define
\[
    \bs{t}^{\bs{a}_j}   :=   t_1^{a_{1j}} t_2^{a_{2j}}\cdots t_\ell^{a_{\ell j}}
\]
so that the action is given by
\[
    \bs{t}\bs{z} = \big( \bs{t}^{\bs{a}_1} z_1, \bs{t}^{\bs{a}_2} z_2, \ldots, \bs{t}^{\bs{a}_n} z_n \big).
\]
Row-reducing $A$ over $\Z$ corresponds to changing coordinates $(t_1,\ldots,t_2)$ for $G$,
so we may assume that $A$ is in reduced echelon form over $\Z$. Similarly, permuting the columns
of $A$ corresponds to reordering the basis for $V_A$.

With respect to the symplectic form given by
$\omega(\bs{z},\bs{z}^\prime) = \Imaginary\langle\bs{z},\bs{z}^\prime\rangle$,
the action of $G$ on $V_A$ is Hamiltonian and admits a unique homogeneous quadratic moment map
$J_A\co V_A\to \mathfrak{g}^\ast$; we will write $J = J_A$ when there is no potential for confusion.
Identifying the Lie algebra $\mathfrak{t}^\ell$ of $\T^\ell$ with
$\R^\ell$ using a basis for $\mathfrak{t}^\ell$ corresponding to the coordinates $(t_1,\ldots,t_\ell)$ for
$\T^\ell$ and the dual basis for $(\mathfrak{t}^\ell)^\ast$, $J = (J_1,\ldots, J_\ell)$ can be expressed in
terms of the component functions
\begin{equation}
\label{eq:MomentMap}
    J_i : V_A \longrightarrow \R,
    \quad\quad
    J_i(\bs{z})
    :=
    \frac{1}{2} \sum\limits_{j=1}^n a_{ij} z_j \overline{z_j},
    \quad\quad
    j=1,\ldots,\ell.
\end{equation}
As the action of $\T^\ell$ on $\mathfrak{t}^\ell$ is trivial, each component  $J_i$ is $\T^\ell$-invariant.
Then the \emph{shell} $Z = Z_A := J^{-1}(0)$ is the $\T^\ell$-stable real algebraic variety in $V_A$
corresponding to this family of quadratics. The \emph{(real) symplectic quotient}
$M_0 = M_0(A) := Z_A/\T^\ell$. The \emph{algebra of smooth functions $\mathcal{C}^\infty(M_0)$} is defined by $\mathcal{C}^\infty(M_0) :=   \mathcal{C}^\infty(V)^G/\mathcal{I}_Z^G$
where $\mathcal{I}_Z$ is the vanishing ideal of $Z$ in $\mathcal{C}^\infty(V)$
and $\mathcal{I}_Z^G := \mathcal{I}_Z\cap\mathcal{C}^\infty(V)^G$.
The algebra $\mathcal{C}^\infty(M_0)$ inherits a Poisson structure
from $\mathcal{C}^\infty(V)$, where the Poisson bracket is given on coordinates by
$\{z_i,\overline{z_j}\} = -2\sqrt{-1}\delta_{ij}$, see \cite{ArmsGotayJennings}.
Equipped with the algebra $\mathcal{C}^\infty(M_0)$ and its Poisson structure,
$M_0$ is a \emph{Poisson differential space}, see \cite[Definition 5]{FarHerSea}.

The algebra of real regular functions $\R[M_0]$ on $M_0$ is defined in terms of the real polynomial invariants $\R[V]^G$. Specifically,
$ \R[M_0] :=   \R[V]^G/I_Z^G$
where $I_Z^G := \mathcal{I}_Z^G\cap\R[V]^G$.
The ideal $I_Z^G$ is homogeneous with respect to
the grading of $\R[V]$ by total degree so that $\R[M_0]$ is a graded algebra; it is as well a Poisson
subalgebra of $\mathcal{C}^\infty(M_0)$. We refer to elements of $\R[V]^G$ as \emph{off-shell invariants}
and the corresponding classes in $\R[M_0]$ as \emph{on-shell invariants}. Note that for $i=1,\ldots,n$,
the real polynomials $z_i \overline{z_i}$ are always invariant. We will take advantage of the complex
coordinate system on $V$ for convenience, often expressing $\R[V]^G$ in terms of polynomials in the
$z_i$ and $\overline{z_i}$. By this, we mean that the real and imaginary parts of these polynomials are
elements of $\R[V]^G$. Note that the real invariants $\R[V]^G$ can be computed in terms of the complexification
$V\otimes_\R \C$ of $V$ by \cite[Proposition 5.8(1)]{GWSLiftingHomotopies}, and $V\otimes_\R \C$ is isomorphic
as a $\T^\ell$-module to $V\oplus V^\ast$.

In this paper, we are primarily interested in the symplectic quotients $M_0(A)$ associated to weight
matrices of a specific form, which we now define.

\begin{definition}
\label{def:TypeAlpha}
We say that an $\ell\times (\ell+k)$ weight matrix $A$ is of \textbf{Type I$\bs{_k}$} if it is of the form
$A = [D, \overbrace{\bs{n},\ldots, \bs{n}}^k]$ where $D = \diag(-a_1, -a_2, \ldots, -a_\ell)$ with
each $a_i > 0$ and $\bs{n} = (n_1, n_2, \ldots, n_\ell)^T$ with each $n_i > 0$. We will say that
$A$ is \textbf{Type II$\bs{_k}$} if $A = [D, c_1\bs{n},\ldots, c_k\bs{n}]$ with $D$ and $\bs{n}$ as above
and each $c_r \geq 1$. Note that a Type I$_k$ weight matrix is Type II$_k$ with each $c_r = 1$.
For a Type II$_k$ weight matrix, we define
\begin{align*}
    \alpha(A)   &:=     \lcm(a_1,\ldots,a_\ell),
    &m_i(A)     &:=     \frac{n_i \alpha(A)}{a_i}
                        \quad\quad\quad\mbox{for}\quad i=1,\ldots, \ell,
    \\
    \beta(A)    &:=     \sum\limits_{i=1}^\ell m_i(A),
                        \quad\quad\quad\mbox{and}\quad
    &\eta(A)    &:= \alpha(A) + \beta(A).
\end{align*}
We will often abbreviate $\alpha(A)$, $m_i(A)$, $\beta(A)$, and $\eta(A)$ as $\alpha$, $m_i$, $\beta$, and $\eta$,
respectively, when $A$ is clear from the context.
\end{definition}

For a weight matrix $A$ of full rank, the representation $V_A$ being faithful is equivalent
to the nonzero $\ell\times\ell$ minors of $A$ having no common factor, see \cite{FarHerSea}.
If $A$ is Type II$_k$, then these minors are of the form $a_1\cdots a_\ell$ or
$a_1\cdots a_{j-1}c_r n_j a_{j+1} \cdots a_\ell$ for some $r=1,\ldots,k$, i.e. the product of the $a_i$ or the
same product with one $a_j$ replaced with $c_r n_j$. The following is an immediate consequence.

\begin{lemma}
\label{lem:Faithful}
Let $A$ be a Type II$_k$ weight matrix. Then $V_A$ is a faithful $\T^\ell$-module
if and only if $\gcd(a_i,a_j) = 1$ for each $1\leq i < j \leq n$, and for each $j = 1,\ldots, \ell$,
there is an $r \leq k$ such that $\gcd(a_j, c_r n_j) = 1$.
\end{lemma}

For a Type I$_k$ or Type II$_k$ weight matrix $A$, the corresponding representation $V_A$ of the
complexification $\T_\C^\ell = (\C^\times)^\ell$ is \emph{stable} and hence \emph{$1$-large}, see
\cite{HerbigSchwarz} for this result and the definitions. Then by \cite[Corollary 4.3]{HerbigSchwarz},
the ideal $I_Z$ is generated by the components $J_i$ of the moment map. Because the $J_i$ are
$G$-invariant in the case under consideration, we have
\[
    \R[M_0] =   \R[V]^G/(J_1,\ldots,J_\ell).
\]
In particular, given Equation~\eqref{eq:MomentMap}, the quotient map $\R[V]^G\to\R[M_0]$
can be understood as defining the invariants $z_i\overline{z_i}$ for $i=1,\ldots,\ell$
in terms of the $z_i\overline{z_i}$ for $i = \ell+1,\ldots,\ell+k$.


\section{Proof of Theorem~\ref{thrm:MainSymplectomorph}}
\label{sec:MainGenProof}

In this section, we give the proof of our main result, Theorem~\ref{thrm:MainSymplectomorph},
which is divided into several auxiliary results. Throughout this section, we consider a Type II$_k$
weight matrix $A = [D, c_1\bs{n},\ldots, c_k\bs{n}] \in\Z^{\ell\times (\ell+k)}$ such that $V_A$ is
a faithful $\T^\ell$-module of dimension $n = \ell + k$. In addition, we let
$B = \big(-\alpha(A), c_1\beta(A),\ldots, c_k\beta(A)\big) \in\Z^{1\times(k+1)}$.
We assume throughout this section that $\ell > 1$; when $\ell = 1$, $A = B$ so that
Theorem~\ref{thrm:MainSymplectomorph} is trivial.

Our first result demonstrates that the $\T^1$-representation $V_B$ is faithful.

\begin{lemma}
\label{lem:GCD(AB)Effective}
Let $A = [D, c_1\bs{n},\ldots, c_k\bs{n}] \in\Z^{\ell\times (\ell+k)}$ be a Type II$_k$ weight matrix.
If $V_A$ is a faithful $\T^\ell$-module, then $\gcd(\alpha(A),c_1\beta(A),\ldots,c_k\beta(A) ) = 1$.
\end{lemma}
\begin{proof}
Suppose $V_A$ is faithful, and let $p$ be a prime that divides $\alpha$ and each $c_r \beta$ for contradiction.
As $p$ divides $\alpha$, it divides some $a_j$; assume $p \mid a_1$ without loss of generality.
By Lemma~\ref{lem:Faithful}, it is not possible that $p\mid c_r$ for all $r$, so it must be that $p\mid\beta$.
Similarly, $p\nmid a_i$ for each $i\neq 1$. Then $p\mid m_i = n_i\alpha/a_i$ for $i > 1$, so the fact that
$p\mid\beta = \sum m_i$ implies that $p\mid m_1$. But as $p$ does not divide any $a_i$ except $a_1$, we have
$\gcd(p,\alpha/a_1) = 1$. Hence $p\mid n_1$.
As $p\mid a_1$ and $p\mid n_1$, $p$ divides the first row of $A$, contradicting the fact that $V_A$ is a
faithful $\T^\ell$-module.
\end{proof}

\begin{lemma}
\label{lem:SympEmbedShell}
The function $\phi\co V_B\to V_A$ defined by
\[
    \phi\co (z_1,\ldots,z_{k+1})\longmapsto
        \left( \sqrt{\frac{m_1}{\beta}}z_1,\sqrt{\frac{m_2}{\beta}}z_1,\ldots,
            \sqrt{\frac{m_\ell}{\beta}}z_1, z_2, z_3, \ldots, z_{k+1}
        \right)
\]
is a symplectic embedding that maps the shell $Z_B = J_{B}^{-1}(0)$ into the shell
$Z_A = J_{A}^{-1}(0)$.
\end{lemma}
\begin{proof}
Using coordinates $(u_1,\ldots,u_n)$ for $V_A$, we have
\begin{align*}
    \phi^\ast \sum\limits_{i=1}^n du_i\wedge d\overline{u_i}
        =    \sum\limits_{i=1}^\ell \frac{m_i}{\beta} d z_1 \wedge d\overline{z_1}
                    + \sum\limits_{i=2}^{k+1} dz_i \wedge d\overline{z_i}
        =    \sum\limits_{i=1}^{k+1} dz_i\wedge d\overline{z_i}
\end{align*}
so that $\phi$ is a symplectic embedding.

Suppose $\bs{z} = (z_1,\ldots,z_{k+1})\in Z_B$ so that
\begin{equation}
\label{eq:MB}
    - \alpha z_1\overline{z_1} + \beta\sum\limits_{j=1}^k c_j z_{j+1}\overline{z_{j+1}} = 0.
\end{equation}
Then for each $i = 1, \ldots, \ell$, we have that
\begin{align*}
    (J_A)_i(\phi(\bs{z}))
        &=  \frac{-a_i m_i}{2\beta} z_1 \overline{z_1} + \frac{n_i}{2}\sum\limits_{j=1}^k c_j z_{j+1} \overline{z_{j+1}}
        \\&=  \frac{-n_i\alpha}{2\beta} z_1 \overline{z_1} + \frac{n_i}{2}\sum\limits_{j=1}^k c_j z_{j+1} \overline{z_{j+1}}
        \\&=  \frac{n_i}{2\beta}\left(-\alpha z_1 \overline{z_1} + \beta\sum\limits_{j=1}^k c_j z_{j+1} \overline{z_{j+1}}
                    \right)
        \quad =  0.
\end{align*}
Hence, $\phi$ maps $Z_B$ into $Z_A$.
\end{proof}

Complexifying the underlying real spaces, we consider the $z_i$ and $w_i:= \overline{z_i}$ as
independent complex coordinates for $V_B\otimes_\R\C$ and $u_i$ and $v_i:= \overline{w_i}$ as
independent complex coordinates for $V_A\otimes_\R\C$.
Let $N_B$ denote the complex shell $(J_B\otimes_\R\C)^{-1}(0)\subset V_B\otimes_\R\C$, i.e.
the set of $(z_1,\ldots,z_{k+1},w_1,\ldots,w_{k+1}) \in V_B\otimes_\R\C$ such that
\begin{equation}
\label{eq:SeshadriSectionNB}
    - \alpha z_1 w_1 + \beta\sum\limits_{j=1}^k c_j z_{j+1}w_{j+1} = 0.
\end{equation}
Similarly, the complex shell $N_A = (J_A\otimes_\R\C)^{-1}(0)\subset V_A\otimes_\R\C$ is defined by
\begin{equation}
\label{eq:SeshadriSectionNA}
    -a_i u_i v_i + n_i\sum\limits_{j=1}^k c_j u_{\ell+j} v_{\ell+j} = 0
    \quad\quad\mbox{for}\quad i=1,\ldots,\ell.
\end{equation}

Recall that if $G$ is a connected algebraic group and $X$ is an irreducible $G$-variety, then a subvariety
$Y \subset X$ is a \emph{Seshadri section} if $\overline{G Y_0} = X$ for each irreducible
component $Y_0$ of $Y$, and $Gy \cap Y = \mathbf{N}(Y)y$ for any $y \in Y$, where
$\mathbf{N}(Y) = \{ g\in G \mid gY = Y \}$. By \cite[Corollary, page 169]{PopovSeshadri} and \cite[Theorem 3.14]{PopovVinberg},
if $X$ is normal, and a Seshadri section $Y$ satisfies $\codim_X \overline{(X\smallsetminus GY)} \geq 2$,
then $Y$ is a Chevalley section, i.e. restriction of functions to $Y$ defines  an isomorphism
$\C[X]^G \to \C[Y]^{\mathbf{N}(Y)}$.

We now demonstrate that these hypotheses are satisfied, i.e. the image of $N_B$ under
$\phi_\C = \phi\otimes_\R\C$ is a Seshadri section for the action of $(\C^\times)^\ell$
on $N_A$.

\begin{lemma}
\label{lem:SheshadriSection}
The image $S:=\phi_\C(N_B)$ of the complex shell $N_B$ is a Seshadri section for the action of $(\C^\times)^\ell$
on the complex shell $N_A\subset V_A\otimes_\R\C$. Moreover, the (complex) codimension of
$\overline{N_A\smallsetminus (\C^\times)^\ell S}$ in $N_A$ is $2$.
\end{lemma}
\begin{proof}
First observe that $S$ is given by the set of points in $V_A\otimes_\R\C$ given by
\begin{align*}
    &\left(  \sqrt{\frac{m_1}{\beta}}z_1,\sqrt{\frac{m_2}{\beta}}z_1,\ldots,
            \sqrt{\frac{m_\ell}{\beta}}z_1, z_2, z_3, \ldots, z_{k+1},
            \right.
            \\ &
            \left. \quad\quad
            \sqrt{\frac{m_1}{\beta}}w_1,\sqrt{\frac{m_2}{\beta}}w_1,\ldots,
            \sqrt{\frac{m_\ell}{\beta}}w_1, w_2, w_3, \ldots, w_{k+1}
    \right)
\end{align*}
for some $z_i$ and $w_i$ that satisfy Equation~\eqref{eq:SeshadriSectionNB}.
As the actions of $\C^\times$ and $(\C^\times)^\ell$ on $V_B\otimes_\R\C$ and $V_A\otimes_\R\C$,
respectively, are stable and hence $1$-large by \cite[Proposition 3.1]{HerbigSchwarz}, both $N_A$
and $N_B$ are reduced and irreducible by \cite[Theorem 2.2 (3)]{HerbigSchwarz}.

Fix a point $(\bs{u},\bs{v}) \in N_A$, i.e. satisfying Equation~\eqref{eq:SeshadriSectionNA},
and assume that each $u_i\neq 0$ for $i \leq \ell$. For $i=2,\ldots,\ell$, choose $t_i$ such that
\[
    t_i^{-a_i}
        =       \sqrt{\frac{m_i}{m_1}} \frac{u_1}{u_i},
    \quad\quad\mbox{i.e.}\quad\quad
    \sqrt{m_1} t_i^{-a_i}u_i
        =       \sqrt{m_i}u_1.
\]
Let $z_1 := u_1\sqrt{\beta/m_1}$, and then
\[
    \sqrt{\frac{m_i}{\beta}}z_1
        =       \sqrt{\frac{m_i}{m_1}} u_1
        =       t_i^{-a_i} u_i.
\]
Similarly, by Equation~\eqref{eq:SeshadriSectionNA}, each $v_i$ with $i = 1,\ldots,\ell$ is given by
\[
    v_i =   \frac{n_i}{a_i u_i}\sum\limits_{j=1}^k c_j u_{\ell+j} v_{\ell+j}.
\]
Letting
\[
    w_1 =   \frac{\sqrt{m_1 \beta}}{\alpha u_1}\sum\limits_{j=1}^k c_j u_{\ell+j} v_{\ell+j},
\]
we have
\[
    v_1 =   \frac{n_1}{a_1 u_1}\sum\limits_{j=1}^k c_j u_{\ell+j} v_{\ell+j}
        =   \frac{m_1}{\alpha u_1}\sum\limits_{j=1}^k c_j u_{\ell+j} v_{\ell+j}
        =   \sqrt{\frac{m_1}{\beta}} w_1,
\]
and, for $i = 2,\ldots,\ell$,
\[
    t_i^{a_i} v_i
        =   \frac{\sqrt{m_1 m_i}}{\alpha u_1}\sum\limits_{j=1}^k c_j u_{\ell+j} v_{\ell+j}
        =   \sqrt{\frac{m_i}{\beta}} w_1.
\]
Hence, letting $\bs{t} = (1,t_2,\ldots,t_\ell)\in(\C^\times)^\ell$ and defining
$z_{i+1} = \bs{t}^{c_i\bs{n}} u_{i+\ell}$ and $w_{i+1} = \bs{t}^{-c_i\bs{n}} v_{i+\ell}$ for $i = 1,\ldots,k$,
we have
\begin{align*}
    \bs{t}(\bs{u},\bs{v})
    &=
    \left(  \sqrt{\frac{m_1}{\beta}} z_1, \ldots,
            \sqrt{\frac{m_\ell}{\beta}} z_1, z_2, z_3, \ldots, z_{k+1},
    \right. \\ & \quad\quad\left.
            \sqrt{\frac{m_1}{\beta}} w_1, \ldots,
            \sqrt{\frac{m_\ell}{\beta}} w_1, w_2, w_3, \ldots, w_{k+1}
    \right).
\end{align*}
Moreover,
\[
    - \alpha z_1 w_1 + \beta\sum\limits_{j=1}^k c_j z_{j+1}w_{j+1}
        =      - \beta \sum\limits_{j=1}^k c_j u_{\ell+j} v_{\ell+j}
                    + \beta\sum\limits_{j=1}^k c_j z_{j+1}w_{j+1}
        = 0,
\]
so that $\bs{t}(\bs{u},\bs{v}) \in S$.
That is, any point $(\bs{u},\bs{v})\in N_A$ with each $u_i\neq 0$ for $i\leq \ell$ is in the
$(\C^\times)^\ell$-orbit of a point in $S$. Note that if each $v_i \neq 0$, then we can define
\[
    t_i^{a_i}
        =       \sqrt{\frac{m_i}{m_1}} \frac{v_1}{v_i}
\]
for $i = 2,\ldots,\ell$ and again obtain $\bs{t}(\bs{u},\bs{v}) \in S$.
Taking the closure to account for points with some $u_i=0$ or $v_i = 0$ for $i\leq\ell$, we have
\begin{equation}
\label{eq:SeshadriSection(1)}
    \overline{(\C^\times)^\ell S} = N_A.
\end{equation}
In particular, note that $N_A\smallsetminus (\C^\times)^\ell S$ consists of those points in $N_A$
where some $u_i = 0$ and some $v_j = 0$ for $i, j \leq \ell$; in particular
$N_A\smallsetminus (\C^\times)^\ell S$ is closed and has codimension $2$ in $N_A$.

Now, recall the definition $\mathbf{N}(S) = \{ \bs{t} \in (\C^\times)^\ell \mid \bs{t}S = S\}$.
We claim that $\mathbf{N}(S) = \{ (t^{\alpha/a_1},\ldots,t^{\alpha/a_\ell}) \mid t\in\C^\times\}$.
Let
\begin{align*}
    (\bs{z},\bs{w})
    &=
    \left(  \sqrt{\frac{m_1}{\beta}}z_1,\ldots,
            \sqrt{\frac{m_\ell}{\beta}}z_1, z_2, z_3, \ldots, z_{k+1},
    \right. \\ & \quad\quad \left.
            \sqrt{\frac{m_1}{\beta}}w_1,\ldots,
            \sqrt{\frac{m_\ell}{\beta}}w_1, w_2, w_3, \ldots, w_{k+1}
    \right)
    \in S,
\end{align*}
and suppose $\bs{t}\in(\C^\times)^\ell$ such that $\bs{t}(\bs{z},\bs{w})\in S$. We have
\begin{align*}
    \bs{t}(\bs{z},\bs{w})
    &=
    \left(  \sqrt{\frac{m_1}{\beta}}t_1^{-a_1} z_1,\ldots,
            \sqrt{\frac{m_\ell}{\beta}} t_\ell^{-a_\ell} z_1, \bs{t}^{c_1\bs{n}}z_2, \ldots,
                \bs{t}^{c_k\bs{n}} z_{k+1},
                 \right.
    \\ &\quad\quad \left.
            \sqrt{\frac{m_1}{\beta}}t_1^{a_1} w_1,\ldots,
            \sqrt{\frac{m_\ell}{\beta}}t_\ell^{a_\ell} w_1, \bs{t}^{-c_1\bs{n}}w_2, \ldots,
                \bs{t}^{-c_k\bs{n}}w_{k+1}
    \right).
\end{align*}
If $z_1 \neq 0$ or $w_1 \neq 0$, we have $t_1^{a_1} = t_i^{a_i}$ for each $i$. Choosing $t\in\C^\times$
such that $t^{\alpha/a_1} = t_1$ and noting that $\gcd(\alpha/a_1,\ldots,\alpha/a_\ell) = 1$ by construction,
it follows that $\bs{t}$ is of the form $(t^{\alpha/a_1},\ldots,t^{\alpha/a_\ell})$. Note that for any such
$\bs{t}$, we have $\bs{t}S = S$ so that
$\mathbf{N}(S) = \{ (t^{\alpha/a_1},\ldots,t^{\alpha/a_\ell}) \mid t\in\C^\times\}$.

If $z_1 = w_1 = 0$, we have $\sum_{j=1}^k c_j z_{j+1} w_{j+1} = 0$. Then
\[
    \bs{t}(\bs{z},\bs{w})
    =
    \left(  0,\ldots, 0, \bs{t}^{c_1\bs{n}}z_2, \ldots, \bs{t}^{c_k\bs{n}} z_{k+1},
            0,\ldots, 0, \bs{t}^{-c_1\bs{n}}w_2, \ldots, \bs{t}^{-c_k\bs{n}}w_{k+1}
    \right).
\]
Choosing an $s\in\C^\times$ such that $s^\beta = \bs{t}^{\bs{n}}$, we have
\begin{align*}
    (s^{\alpha/a_1},\ldots,s^{\alpha/a_\ell})(\bs{z},\bs{w})
    &=
    \big(  0,\ldots, 0, s^{c_1\sum_i n_i\alpha/a_i}z_2, \ldots, s^{c_k\sum_i n_i\alpha/a_i} z_{k+1},
    \\ &\quad\quad
            0,\ldots, 0, s^{-c_1\sum_i n_i\alpha/a_i}w_2, \ldots, s^{-c_k\sum_i n_i\alpha/a_i}w_{k+1}
    \big)
    \\ &=
    \big(  0,\ldots, 0, s^{c_1\beta}z_2, \ldots, s^{c_k\beta} z_{k+1},
    \\ &\quad\quad
            0,\ldots, 0, s^{-c_1\beta}w_2, \ldots, s^{-c_k\beta}w_{k+1}
    \big)
    \quad =
    \bs{t}(\bs{z},\bs{w})
\end{align*}
so that $(\C^\times)^\ell (\bs{z},\bs{w}) \subset \mathbf{N}(S)(\bs{z},\bs{w})$.
\end{proof}

As $S$ is a Seshadri section for the action of $(\C^\times)^\ell$ on $N_A$ such that the codimension
of $\overline{N_A\smallsetminus (\C^\times)^\ell S}$ in $N_A$ is $2$, we have that
the restriction of functions to $S$ defines an isomorphism
$\C[N_A]^{(\C^\times)^\ell}  \to \C[S]^{\mathbf{N}(S)}$ by
\cite[Corollary, page 169]{PopovSeshadri}; see also \cite[Theorem 3.14]{PopovVinberg}.
Note that $\mathbf{N}(S)$ acts on the subspace of $V_A$ spanned by
$(\overbrace{1,\ldots,1}^\ell,0,\ldots,0)$ and the standard unit vectors $e_i$ for $i > \ell$
with weight vector $(-\alpha, c_1\beta,\ldots, c_k\beta)$.
Then as $S$ is isomorphic to the shell $N_B$ via the embedding $\phi_\C$,
it follows that $\phi_\C^\ast$ induces an isomorphism
$\phi_\C^\ast\co\C[S]^{\mathbf{N}(S)} \to \C[N_B]^{\C^\times}$.
As $\phi_\C$ is a linear map, $\phi_\C^\ast$ preserves the grading.
Then by \cite[Lemma 2.5]{HerbigSchwarzSeaton2}, as the representations
of $(\C^\times)^\ell$ and $\C^\times$ corresponding to $A$ and $B$, respectively,
are $1$-large, we have that
$\R[Z_A]^{\T^\ell} \otimes\C \simeq \C[N_A]^{(\C^\times)^\ell}$ and
$\R[Z_B]^{\T^1} \otimes\C \simeq \C[N_B]^{\C^\times}$. That is,
$\phi^\ast$ induces a graded isomorphism of the algebras of real
regular functions $\R[M_0(A)] \to \R[M_0(B)]$. By Lemma~\ref{lem:SympEmbedShell},
this isomorphism is Poisson.

Summarizing, we have the following.

\begin{corollary}
\label{cor:Restriction}
The restriction of functions to $S$ and pulling back via $\phi_\C$
are both graded isomorphisms
\[
    \C[N_B]^{\C^\times}
    \overset{\phi_\C^\ast}{\longrightarrow}
    \C[S]^{\mathbf{N}(S)}
    \longrightarrow
    \C[N_A]^{(\C^\times)^\ell},
\]
and the composition of these maps induces a graded Poisson isomorphism of
the real algebras
\[
    \Psi\co\R[M_0(A)]
    \longrightarrow
    \R[M_0(B)].
\]
\end{corollary}

By Lemmas~\ref{lem:SympEmbedShell} and~\ref{lem:SheshadriSection} and Corollary~\ref{cor:Restriction},
it follows that $\phi$ induces an isomorphism between the Zariski closures of the real algebraic varieties
defined by $\R[Z_A]^{\T^\ell}$ and $\R[Z_B]^{\T^1}$.
To complete the proof of Theorem~\ref{thrm:MainSymplectomorph}, it remains only to show that
the semialgebraic conditions are preserved, i.e. the map $\phi$ induces a homeomorphism
between the symplectic quotients.

\begin{lemma}
\label{lem:Semialg}
The map $\phi$ induces a homeomorphism $M_0(B) = Z_B/\T^1 \to M_0(A) = Z_A/\T^\ell$.
\end{lemma}
\begin{proof}
It is clear that $\phi$ maps $\T^1$-orbits into $\T^\ell$-orbits, as if
$\bs{z} = (z_1,\ldots,z_{k+1})\in Z_B$ and $t \in \T^1$, then
\begin{align*}
    \phi(t\bs{z})
        &=      \phi( t^{-\alpha} z_1, t^{c_1\beta} z_2, \ldots, t^{c_k\beta} z_{k+1})
        \\&=    \left( \sqrt{\frac{m_1}{\beta}}t^{-\alpha} z_1, \ldots,
                    \sqrt{\frac{m_\ell}{\beta}} t^{-\alpha} z_1, t^\beta z_2, \ldots, t^\beta z_{k+1}
                \right)
        \\&=    \left( \sqrt{\frac{m_1}{\beta}}(t^{\alpha/a_1})^{-a_1} z_1, \ldots,
                    \sqrt{\frac{m_\ell}{\beta}} (t^{\alpha/a_\ell})^{-a_\ell} z_1, \right.
        \\&\quad\quad\quad\quad
                \left.
                    (t^{\alpha/a_1})^{c_1 n_1}\cdots(t^{\alpha/a_\ell})^{c_1 n_\ell} z_2, \ldots,
                    (t^{\alpha/a_1})^{c_k n_1}\cdots(t^{\alpha/a_\ell})^{c_k n_\ell} z_{k+1}
                \right)
        \\&=    (t^{\alpha/a_1},\cdots,t^{\alpha/a_\ell}) \phi(\bs{z}).
\end{align*}
As $\phi(Z_B) \subset (Z_A)$ by Lemma~\ref{lem:SympEmbedShell}, it is sufficient to show that
each element of $Z_A$ is in the orbit of an element of $\phi(Z_B)$. So let
$\bs{u} = (u_1,\ldots,u_n) \in Z_A$ so that for $i=1,\ldots,\ell$,
\[
    -a_i u_i \overline{u_i} + n_i\sum\limits_{j=1}^k c_j u_{\ell+j} \overline{u_{\ell+j}} = 0,
    \quad\quad\mbox{i.e.}\quad\quad
    \frac{a_i}{n_i} u_i \overline{u_i} = \sum\limits_{j=1}^k c_j u_{\ell+j} \overline{u_{\ell+j}}.
\]
As each $a_i, n_i, c_j > 0$, it follows that if $u_i = 0$ for some $i\leq\ell$, then
$u_i = 0$ for each $i > \ell$, i.e. $\bs{u} = \bs{0} = \phi(\bs{0})$.
Hence, we may assume each $u_i$ is nonzero. Then for $i = 2,\ldots,\ell$, we have
\[
    |u_i|   =       \sqrt{\frac{a_1 n_i}{a_i n_1}} |u_1|
            =       \sqrt{\frac{m_i}{m_1}} |u_1|.
\]
Hence for $i = 2,\ldots,\ell$, there is a $t_i\in\T^1$ such that
\[
    t_i^{-a_i} u_i
        =       \sqrt{\frac{m_i}{m_1}} u_1.
\]
Then setting $\bs{t} := (1,t_2,\ldots,t_\ell)$, $z_1 := u_1\sqrt{\beta/m_1}$, and
$z_{i+1} := \bs{t}^{c_i \bs{n}} u_{\ell+i}$ for $i > 1$, we have that
\begin{align*}
    \bs{t}(u_1,\ldots,u_n)
        &=          (u_1, t_2^{-a_2}u_2, \ldots, t_\ell^{-a_\ell} u_\ell,
                        \bs{t}^{c_1\bs{n}} u_{\ell+1}, \ldots, \bs{t}^{c_k\bs{n}} u_n)
        \\&=        \left( \sqrt{\frac{m_1}{\beta}}z_1,\sqrt{\frac{m_2}{\beta}}z_1,\ldots,
                        \sqrt{\frac{m_\ell}{\beta}}z_1, z_2, z_3, \ldots, z_{k+1}
                    \right)
        \quad =          \phi(z_1,\ldots,z_{k+1}).
\end{align*}
Finally, we note that $(z_1,\ldots,z_{k+1})$ satisfy Equation~\eqref{eq:MB}, as
\begin{align*}
    - \alpha z_1\overline{z_1} + \beta\sum\limits_{j=1}^k c_j z_{j+1}\overline{z_{j+1}}
    &=      - \frac{\beta \alpha}{m_1} u_1 \overline{u_1}
                + \beta\sum\limits_{j=1}^k c_j u_{\ell+j}\overline{u_{\ell+j}}
    \\&=    \frac{\beta}{n_1} \left(- a_1 u_1 \overline{u_1}
                + n_1\sum\limits_{j=1}^k c_j u_{\ell+j}\overline{u_{\ell+j}}
            \right)
    \quad =  0,
\end{align*}
so that $\bs{t}(u_1,\ldots,u_n) \in \phi(Z_B)$. It follows that each $\T^\ell$-orbit in
$Z_A$ intersects $\phi(Z_B)$.

We leave it to the reader to show that the inverse homeomorphism is induced by the linear map
\[(u_1,u_2,\dots,u_{k+ \ell}) \mapsto(\sqrt{\beta\over m_1}u_1,u_{\ell+1},\dots, u_{k+ \ell}).\]
\end{proof}

We illustrate Theorem~\ref{thrm:MainSymplectomorph} with the following.

\begin{example}
\label{ex:Gen4}
The weight matrix
\[
    A   =   \begin{pmatrix}
                -3  &   0   &   0   &   1   &   2   &   3   &   3   \\
                0   &   -4  &   0   &   3   &   6   &   9   &   9   \\
                0   &   0   &   -5  &   2   &   4   &   6   &   6
            \end{pmatrix}
\]
is Type II$_4$ with $\alpha = 60$, $n_1 = 1$, $n_2 = 3$, $n_3 = 2$,
$c_1 = 1$, $c_2 = 2$, and $c_3 = c_4 = 3$. Hence, $m_1 = 20$, $m_2 = 45$, $m_3 = 24$,
and $\beta = 89$, and the symplectic quotient $M_0(A)$
is graded regularly symplectomorphic to that associated to $(-60, 89, 178, 267, 267)$.
\end{example}


\section{Constructive Approach to Theorem~\ref{thrm:MainSymplectomorph}}
\label{sec:ConstrucProof}

We first obtained a proof of Theorem~\ref{thrm:MainSymplectomorph} for Type I$_k$
matrices by determining an explicit description of the symplectic quotient $M_0$
and algebra $\R[M_0]$ of regular functions. This description may be of independent
interest and illustrates the structure of these spaces, so we include it here.
The proofs of these results are cumbersome computations and hence only summarized.

\begin{proposition}
\label{prop:ConstrucHBasis}
Let $A = [D, \overbrace{\bs{n},\ldots, \bs{n}}^k]\in\Z^{\ell\times(\ell+k)}$ be a type I$_k$ weight matrix
such that $V_A$ is a faithful $\T^\ell$-module. Then a generating set for the algebra
$\R[V_A]^{\T^\ell}$ of invariants is given by
\begin{enumerate}
\item   the $\ell$ quadratic monomials $r_i := z_i \overline{z_i}$ for $i = 1,\ldots,\ell$,
\item   the $k^2$ quadratic monomials $p_{i,j} := z_{\ell+i}\overline{z_{\ell+j}}$ for $1\leq i,j\leq k$,
\item   the ${\alpha+k-1 \choose k-1}$ degree $\eta$ monomials
        $q_{\bs{s}} := \prod_{i=1}^\ell z_i^{m_i} \prod_{i=1}^{k} z_{\ell+i}^{s_i}$ where
        $\bs{s} = (s_1,\ldots,s_k)$ and the $s_i$ are any choice of nonnegative integers
        such that $\sum_{i=1}^{k} s_i = \alpha$, and
\item   the ${\alpha+k-1 \choose k-1}$ degree $\eta$ monomials $\overline{q_{\bs{s}}}$ for each choice of $\bs{s}$.
\end{enumerate}
For a generating set for $\R[M_0(A)]$, the generators in (1) can be omitted using the
on-shell relations.
\end{proposition}

A simple computation demonstrates that each of the monomials listed in Proposition~\ref{prop:ConstrucHBasis}
is invariant.
To prove the proposition, one first establishes the result when $k=1$ by induction on $\ell$;
the base case is simple, and the inductive step is accomplished by comparing the invariants
of $A$ to those corresponding to submatrices formed by removing a single row and the resulting
column of zeros. For general $k$, consider the map
$\phi\co\R[z_1,\ldots,z_{\ell+k},\overline{z_1},\ldots,\overline{z_{\ell+k}}]
\to \R[w_1,\ldots, w_{\ell+1},\overline{w_1},\ldots, \overline{w_{\ell+1}}]$ that maps
$z_i\mapsto w_i$ and $\overline{z_i}\mapsto\overline{w_i}$ for $i \leq \ell$,
$z_{\ell+i}\mapsto w_{\ell+1}$, and $\overline{z_{\ell+i}}\mapsto \overline{w_{\ell+1}}$.
It is easy to see that $\phi$ maps $A$-invariants onto $[D, \mathbf{n}]$-invariants,
and then the proof is completed by considering the preimages of the $[D, \mathbf{n}]$-invariants,
a case with $k=1$.

\begin{proposition}
\label{prop:ConstrucRelations}
Let $A = [D, \overbrace{\bs{n},\ldots, \bs{n}}^k]\in\Z^{\ell\times(\ell+k)}$ be a type I$_k$ weight matrix
such that $V_A$ is a faithful $\T^\ell$-module. The (off-shell) relations among the $r_i$, $p_{i,j}$,
$q_{\bs{\alpha}}$, and $\overline{q_{\bs{\alpha}}}$ are generated by the following.
\begin{enumerate}
\item   $p_{g,h}p_{i,j} - p_{g,j}p_{i,h}$ for $1\leq g,h,i,j \leq k$ with
        $g\neq i$ and $h\neq j$.
\item   $p_{g,h} q_{\bs{s}} - p_{i,h} q_{\bs{s^\prime}}$ where $s_g^\prime = s_g + 1$,
        $s_i^\prime = s_i - 1$, and $s_j^\prime = s_j$ for $j \neq g,i$.
        Note that we must have $s_i \geq 1$.
\item   $p_{g,h} \overline{q_{\bs{s}}} - p_{g,i} \overline{q_{\bs{s^\prime}}}$
        where $s_g^\prime = s_g + 1$,
        $s_i^\prime = s_i - 1$, and $s_j^\prime = s_j$ for $j \neq g,i$.
        Note that we must have $s_i \geq 1$.
\item   $q_{\bs{s}} q_{\bs{s^\prime}} - q_{\bs{t}} q_{\bs{t^\prime}}$
        where $\bs{s} + \bs{s^\prime} = \bs{t} + \bs{t^\prime}$ and $\bs{s}\neq\bs{t}$.
\item   $\overline{q_{\bs{s}}}\, \overline{q_{\bs{s^\prime}}}
            - \overline{q_{\bs{t}}}\, \overline{q_{\bs{t^\prime}}}$
        where $\bs{s} + \bs{s^\prime} = \bs{t} + \bs{t^\prime}$ and $\bs{s}\neq\bs{t}$.
\item   $\prod_{i=1}^\ell r_i^{m_i} \prod_{j=1}^\alpha p_{g_j,h_j}
            - q_{\bs{s}} \overline{q_{\bs{s^\prime}}}$ where the vector $(g_1, \ldots, g_\alpha)$
            contains each value $g$ exactly $s_g$ times and the vector
            $(h_1,\ldots, h_\alpha)$ contains each value $h$ exactly $s_h^\prime$ times.
\end{enumerate}
On-shell, the monomials additionally satisfy the defining relations of the moment map,
$-a_i r_i + n_i \sum_{j=1}^k p_{j,j}$ for $i=1,\ldots,\ell$.
\end{proposition}

One verifies that each of these relations holds by direct computation using the definitions
of the monomials given in Proposition~\ref{prop:ConstrucHBasis}. The proof that all relations
are generated by these is by induction on $k$. For the case $k = 1$, there is only one nontrivial
relation, $p_{1,1}^\alpha \prod_{i=1}^\ell r_i^{m_i} - q_{(\alpha)} \overline{q_{(\alpha)}}$;
a simple yet tedious consideration of cases demonstrates that this generates all relations.
The induction step is demonstrated by considering the preimages of invariants under the map
$\C[z_1,\ldots,z_{\ell+k+1}]\to\C[z_1,\ldots,z_{\ell+k}]$ given by
$(z_1,\ldots,z_{\ell+k+1})\mapsto(z_1,\ldots,z_{\ell+k} + z_{\ell+k+1})$.

One then verifies the following by direct computation.

\begin{proposition}
\label{prop:ConstrucPoisson}
Let $A = [D, \overbrace{\bs{n},\ldots, \bs{n}}^k]\in\Z^{\ell\times(\ell+k)}$ be a type I$_k$ weight matrix
such that $V_A$ is a faithful $\T^\ell$-module. The Poisson brackets of the Hilbert basis elements
given in Proposition~\ref{prop:ConstrucHBasis} are as follows. Note that the indices
$g,h,i,j$ need not be distinct unless otherwise noted.
\begin{itemize}
\item   $\{ r_g, r_h\} = \{ r_g, p_{h,i}\} = \{ q_{\bs{s}}, q_{\bs{s^\prime}} \}
            = \{ \overline{q_{\bs{s}}}, \overline{q_{\bs{s^\prime}}} \} = 0$.
\item   $\{ r_i, q_{\bs{s}} \} = - \frac{2}{\sqrt{-1}} m_i q_{\bs{s}}$.
\item   $\{ r_i, \overline{q_{\bs{s}}} \} = \frac{2}{\sqrt{-1}} m_i \overline{q_{\bs{s}}}$.
\item   $\{ p_{g,h}, p_{i,j}\} = \begin{cases}
            \frac{2}{\sqrt{-1}} p_{i,h},   &   \mbox{$g = j$ and $h\neq i$},
            \\
            - \frac{2}{\sqrt{-1}} p_{g,j}, &   \mbox{$g \neq j$ and $h = i$},
            \\
            \frac{2}{\sqrt{-1}}(p_{h,h} - p_{g,g})  &
                                \mbox{$g = j$ and $h = i$, and $g\neq h$}
            \\
            0,      &   \mbox{$g \neq j$ and $h\neq i$ or $g = j = h = i$}.
            \end{cases}$.
\item   $\{ p_{g,h}, q_{\bs{s}} \} = \begin{cases}
            -\frac{2}{\sqrt{-1}} s_g q_{\bs{s^\prime}},   &   s_g > 0,
            \\
            0,      &   s_g = 0,\end{cases}$ \\ where
            $s_g^\prime = s_g - 1$, $s_h^\prime = s_h + 1$, and $s_i^\prime = s_i$
            for $i\neq g, h$.
\item   $\{ p_{g,h}, \overline{q_{\bs{s}}} \} = \begin{cases}
            \frac{2}{\sqrt{-1}} s_g \overline{q_{\bs{s^\prime}}},   &   s_g > 0,
            \\
            0,      &   s_g = 0,\end{cases}$ \\ where
            $s_g^\prime = s_g - 1$, $s_h^\prime = s_h + 1$, and $s_i^\prime = s_i$
            for $i\neq g, h$.
\item   $\{ q_{\bs{s}}, \overline{q_{\bs{s^\prime}}} \} = \frac{2}{\sqrt{-1}}
        q_{\bs{s}} \overline{q_{\bs{s^\prime}}} \left(\sum_{i=1}^\ell \frac{m_i^2}{r_i}
            + \sum_{j=1}^k \frac{s_j s_j^\prime}{p_{j,j}}\right)$, which we note
        is polynomial as the $r_i$ and $p_{j,j}$ divide $q_{\bs{s}} q_{\bs{s^\prime}}$.
\end{itemize}
\end{proposition}

The above results give an explicit description of the Poisson algebra of regular functions. It remains
only to determine the semialgebraic description of the symplectic quotient.

\begin{proposition}
\label{prop:ConstrucInequalities}
Let $A = [D, \overbrace{\bs{n},\ldots, \bs{n}}^k]\in\Z^{\ell\times(\ell+k)}$ be a type I$_k$ weight matrix associated
such that $V_A$ is a faithful $\T^\ell$-module. Using the real Hilbert basis given by the real and imaginary parts of
the monomials listed in Proposition~\ref{prop:ConstrucHBasis}, the image of the Hilbert embedding is described
by the relations given in Proposition~\ref{prop:ConstrucRelations} as well as the
inequalities $r_i \geq 0$ for $i = 1,\ldots, \ell$ and $p_{j,j} \geq 0$ for
$j = 1, \ldots, k$.
\end{proposition}

From the definition of the monomials, it is easy to see that these inequalities are satisfied.
For the converse, choose values of the $r_i$, $p_{i,j}$, and $q_{\bs{s}}$ such that each $r_i \geq 0$,
each $p_{i,i} \geq 0$, and the remaining values are arbitrary elements of $\C$ such that the each
$p_{i,j} = \overline{p_{j,i}}$ and relations in Proposition~\ref{prop:ConstrucRelations} are satisfied.
It is then easy to see that the values $|r_i|$, $|p_{i,j}|$ for $i\neq j$, and $|q_{\bs{s}}|$ are determined
by the $p_{i,i}$.  Specifically, using the relations of Proposition~\ref{prop:ConstrucRelations}(1), we have
\[
    |p_{i,j}|   =   \sqrt{p_{i,i} p_{j,j}},
\]
using the moment map, we have
\[
    |r_i|       =   \frac{n_i}{a_i} \sum\limits_{j=1}^k p_{j,j}
\]
and using the relations of Proposition~\ref{prop:ConstrucRelations}(6), we have
\[
    q_{\bs{s}}  =   \sqrt{\prod\limits_{i=1}^\ell \left(\frac{n_i}{a_i}\right)^{m_i}
                        \left(\sum\limits_{j=1}^k p_{i,i}\right)^{\sum_{i=1}^\ell m_i} }
                        \left(\prod\limits_{j=1}^k p_{i,i}^{s_i}\right)^{\alpha/2}.
\]
Similarly, using the relations of Proposition~\ref{prop:ConstrucRelations}(3), one checks that the
arguments of the $q_{\bs{s}}$ where $\bs{s}$ has only one nonzero coordinate (which must be equal to $\alpha$)
determine the arguments of the $p_{i,j}$ and the other $q_{\bs{s^\prime}}$. It follows that one can find a
point $(z_1,\ldots,z_n)$ mapped via the Hilbert embedding to these values of $r_i$, $p_{i,j}$, and $q_{\bs{s}}$
by choosing the modulus of each $z_{\ell+i}$ to be $\sqrt{p_{i,i}}$, the modulus of each
$z_i$ for $i\leq\ell$ to be determined by the moment map, the argument of each $z_i$ for $i\leq\ell$ to be $0$,
and the argument of each $z_{\ell+i}$ to be the argument of $q_{(0,\ldots,0,\alpha,0,\ldots,0)}$ where $\alpha$
occurs in the $i$th position.

With this explicit description of $M_0(A)$ and $\R[M_0(A)]$ the following can be verified by explicit computation.

\begin{theorem}
\label{thrm:Construc}
Let $A\in\Z^{\ell\times (\ell+k)}$ be a Type I$_k$ matrix such that $V_A$ is a faithful $\T^\ell$-module,
and let $B = \big(-\alpha(A), c_1\beta(A),\ldots, c_k\beta(A)\big) \in\Z^{1\times(k+1)}$.
Using coordinates $(w_1,\ldots,w_{k+1})$ for $V_B$, define the map
$\Phi\co\C[V_A]^{\T^\ell} \to \C[V_B]$ by
\begin{align*}
    r_i             &\longmapsto    \frac{m_i(A)}{\beta(A)} w_1 \overline{w_1},  \quad\quad\quad 1\leq i \leq \ell,
    \\
    p_{ij}          &\longmapsto    w_{i+1}\overline{w_{i+1}},              \quad\quad\quad 1\leq i,j \leq k,
    \\
    q_{\bs{s}}      &\longmapsto    \sqrt{\beta(A)^{-\beta(A)} \prod\limits_{j=1}^\ell m_j(A)^{m_j(A)}}
                                    \quad w_1^{\beta(A)} \prod\limits_{j=1}^k w_{j+1}^{s_j},
    \\
    \overline{q_{\bs{s}}}
                    &\longmapsto    \sqrt{\beta(A)^{-\beta(A)} \prod\limits_{j=1}^\ell m_j(A)^{m_j(A)}}
                                    \quad \overline{w_1}^{\beta(A)}
                                    \prod\limits_{j=1}^k \overline{w_{j+1}}^{s_j}.
\end{align*}
Then $\Phi$ is a well-defined homomorphism $\Phi\co\C[V_A]^{\T^\ell} \to \C[V_B]^{\T^1}$
inducing an isomorphism $\R[M_0(A)]\to\R[M_0(B)]$ and a graded regular symplectomorphism between
$M_0(A)$ and $M_0(B)$.
\end{theorem}


\section{Classification for Type I$_k$ matrices}
\label{sec:Classification}

In the case $k = 1$, Corollary~\ref{cor:SIGMASympCrit} implies that two weight matrices $A_1$ and $A_2$
yield graded regularly symplectomorphic symplectic quotients if and only if $\eta(A_1) = \eta(A_2)$,
i.e. if and only if $\alpha(A_1)+\beta(A_1) = \alpha(A_2)+\beta(A_2)$. For $k > 1$, this
is no longer the case, as we demonstrate with the following.

\begin{lemma}
\label{lem:TypeINoOthers}
Let $A = (-\alpha, \overbrace{\beta,\ldots,\beta}^k)$
and $B = (-\alpha^\prime, \overbrace{\beta^\prime,\ldots,\beta^\prime}^{k^\prime})$
such that $V_A$ and $V_B$ are faithful $\T^1$-modules. If the symplectic quotients
$M_0(A)$ and $M_0(B)$ are graded regularly diffeomorphic for $k\geq 2$, then
$k = k^\prime$, $\alpha = \alpha^\prime$ and $\beta = \beta^\prime$.
\end{lemma}
\begin{proof}
First note that the fact that $V_A$ and $V_B$ are faithful implies that
$\gcd(\alpha,\beta) = \gcd(\alpha^\prime,\beta^\prime) = 1$.
The existence of a graded regular diffeomorphism implies that $\R[M_0(A)]$ is graded isomorphic
to $\R[M_0(B)]$. As the Krull dimensions of $\R[M_0(A)]$ and $\R[M_0(B)]$ are given by $2k$ and
$2k^\prime$, respectively, it follows that $k = k^\prime$.

Let $\mathcal{Q}(A)$ denote the subalgebra of $\R[M_0(A)]$ that is generated by the quadratic
monomials of the form $z_i \overline{z_i} + I_{Z_A}^G$ for $i=1\,\ldots,k+1$
and $z_{1+i} \overline{z_{1+j}}+ I_{Z_A}^G$ for $1\leq i,j \leq k$, and define
$\mathcal{Q}(B)$ identically as a subalgebra of $\R[M_0(B)]$.
Note that $\mathcal{Q}(A)$ and $\mathcal{Q}(B)$ are obviously graded isomorphic.
The lowest-degree element of $\R[M_0(A)]$ that is not an element of $\mathcal{Q}(A)$
has degree $\alpha+\beta$, and similarly for $\R[M_0(B)]$, so we can conclude that
$\alpha+\beta = \alpha^\prime+\beta^\prime$.
Finally, the number of monomials in $\R[M_0(A)]$ of degree $\alpha+\beta$ that are
not elements of $\mathcal{Q}(A)$ is ${\alpha+k-1\choose k-1}$, and hence
${\alpha+k-1\choose k-1}={\alpha^\prime+k-1\choose k-1}$, i.e.
$(\alpha+k-1)!/\alpha! = (\alpha^\prime+k-1)!/\alpha^\prime!$. As $k > 1$,
it follows that $\alpha = \alpha^\prime$, and hence $\beta = \beta^\prime$.
\end{proof}

\begin{corollary}
\label{cor:ClassifyTypeI}
The graded regular symplectomorphism classes of symplectic quotients associated to Type I$_k$
weight matrices with $k > 1$ are classified by the triple $(k,\alpha(A),\beta(A))$. Moreover, these
graded regular symplectomorphism classes coincide with the graded regular diffeomorphism classes.
\end{corollary}

It is not clear whether an analog to Lemma~\ref{lem:TypeINoOthers} is true for
Type II$_k$ matrices, but a proof using only the grading of $\R[M_0]$ as in
Lemma~\ref{lem:TypeINoOthers} is not possible.
First note that such a generalization would require restricting to
specific representatives, e.g. requiring that $\gcd(c_1,\ldots,c_k) = 1$.
Otherwise, it is possible that a $1\times (k+1)$ Type II$_k$ matrix could be written in terms of
$\alpha$, $\beta$, and the $c_i$ in more than one way, e.g. $(-1, 4, 12)$ could correspond to
$\alpha = 1$, $\beta = 2$, $c_1 = 2$, and $c_2 = 6$ or to $\alpha = 1$, $\beta = 4$, $c_1 = 1$, and $c_2 = 3$.
However, even with such a restriction, it is possible that $\R[M_0(A)]$ and $\R[M_0(B)]$ have the same Hilbert
series yet fail to be graded regularly symplectomorphic. We will illustrate this in the next section.


\section{The Hilbert series does not classify symplectic quotients by tori}
\label{sec:HilbCounterex}

The graded regular symplectomorphisms given by Theorem~\ref{thrm:MainSymplectomorph} were initially discovered
by computing Hilbert series of the algebras of regular functions on symplectic quotients
associated to large classes of weight matrices and looking for cases that coincide. While the Hilbert
series has been a valuable heuristic to indicate potential graded regular symplectomorphisms
and an important tool to distinguish between non-graded regularly symplectomorphic cases,
one would likely guess that there are cases with the same Hilbert series that are not graded
regularly symplectomorphic. In this section, we give examples to indicate that this is the case:
the Hilbert series is not a fine enough invariant to distinguish graded regular symplectomorphism
classes of symplectic quotients by tori. These examples further illustrate that two symplectic
quotients can have several isomorphic structures yet fail to be graded regularly symplectomorphic.

Let $A = (-2,3,6)$ and $B = (-3,2,6)$. Note that these are both Type II$_2$ weight matrices;
$A$ corresponding to $\alpha = 2$, $\beta = 3$, $c_1 = 1$, and $c_2 = 2$; and
$B$ corresponding to $\alpha = 3$, $\beta = 2$, $c_1 = 1$, and $c_2 = 3$).
Because the Hilbert series of symplectic quotients by
$\T^1$ only depends on the sign of the weights (see \cite[page 47]{HerbigSeatonHilbSympCirc}),
it is clear that the Hilbert series of $\R[M_0(A)]$ and $\R[M_0(B)]$ coincide. In particular,
they are both given by
\[
    \frac{1 + t^3 + 2t^4 + t^5 + t^8}
        {(1 - t^5)(1 - t^3)(1 - t^2)^3}.
\]
The off-shell invariants $\R[V_A]^{\T^1}$ are generated by
\begin{align*}
    p_0 = &z_1 \overline{z_1},\quad
    p_1 = z_2 \overline{z_2},\quad
    p_2 = z_3 \overline{z_3},\quad
    p_3 = z_2^2 \overline{z_3},\quad
    p_4 = z_3 \overline{z_2}^2,
    \\
    &p_5 = z_1^3 z_3,\quad
    p_6 = \overline{z_1}^3 \overline{z_3},\quad
    p_7 = z_1^3 z_2^2,\quad
    p_8 = \overline{z_1}^3 \overline{z_2}^2 ,
\end{align*}
and the moment map determines $p_0$ via $2p_0 = 3p_1 + 6p_2$.
The off-shell invariants $\R[V_B]^{\T^1}$ are generated by
\begin{align*}
    q_0 = &u_1 \overline{u_1},\quad
    q_1 = u_2 \overline{u_2},\quad
    q_2 = u_3 \overline{u_3},\quad
    q_3 = u_1^2 u_3,\quad
    q_4 = \overline{u_1}^2 \overline{u_3},
    \\
    &q_5 = u_2^3 \overline{u_3},\quad
    q_6 = u_3 \overline{u_2}^3,\quad
    q_7 = u_1^2 u_2^3,\quad
    q_8 = \overline{u_1}^2 \overline{u_2}^3,
\end{align*}
and the shell relation is given by $3q_0 = 2q_1 + 6q_2$.

\begin{proposition}
\label{prop:HilbCounter}
For the weight matrices $A = (-2,3,6)$ and $B = (-3,2,6)$, the following hold true.
\begin{itemize}
\item[({\it i.})]       The algebras $\R[M_0(A)]\otimes_{\R}\C$ and $\R[M_0(B)]\otimes_{\R}\C$
                        are graded Poisson isomorphic. Hence, the complex symplectic quotients
                        are isomorphic as Poisson varieties.
\item[({\it ii.})]      The algebras $\R[M_0(A)]$ and $\R[M_0(B)]$ are graded isomorphic.
                        However, no graded isomorphism $\R[M_0(A)]\to\R[M_0(B)]$ preserves the
                        inequalities describing the semialgebraic sets $M_0(A)$ and $M_0(B)$.
\end{itemize}
\end{proposition}

An immediate consequence of ({\it ii.}) is that the symplectic quotients
$M_0(A)$ and $M_0(B)$ are not graded regularly symplectomorphic.

\begin{proof}[Proof of Proposition~\ref{prop:HilbCounter}({\it i.})]
As in the proof of Lemma~\ref{lem:SheshadriSection}, we complexify the underlying real vector spaces to
consider the $z_i$, $w_i:= \overline{z_i}$, $u_i$, and $v_i:= \overline{w_i}$ as independent complex variables.
Then an easy-to-identify isomorphism over $\C$ is induced by the linear map
$\phi\co V_A\otimes_\R\C \to V_B\otimes_\R\C$ given by
\[
    \phi\co (z_1, z_2, z_3, w_1, w_2, w_3)
    \longmapsto
    (\sqrt{-1} w_2, \sqrt{-1} w_1, z_3, \sqrt{-1} z_2, \sqrt{-1} z_1, w_3 ).
\]
A simple computation demonstrates that $\phi$ is equivariant with respect to the two $\C^\times$-actions,
implying that the corresponding map
$\phi^\ast\co\C[V_B\otimes_\R\C]^{\C^\times}\to\C[V_A\otimes_\R\C]^{\C^\times}$
is an isomorphism. Using coordinates $(u_1,u_2,u_3,v_1,v_2,v_3)$ for $V_B\otimes_\R\C$, we have
\begin{align*}
    \phi^\ast ( du_1\wedge dv_1 + du_2\wedge dv_2 + du_3\wedge dv_3 )
        &=      - dw_2 \wedge dz_2 - dw_1 \wedge dz_1 + dz_3 \wedge dw_3
        \\&=    dz_1 \wedge dw_1 + dz_2 \wedge dw_2 + dz_3 \wedge dw_3
\end{align*}
so that $\phi$ is a symplectic embedding.

Identifying the real and complex invariants via $w_i = \overline{z_i}$ and $v_i= \overline{w_i}$,
the map $\phi^\ast$ is given on generators by
\begin{align*}
    \phi^\ast q_0   &=  - z_2 w_2 = - p_1,
    &
    \phi^\ast q_1   &=  - z_1 w_1 = - p_0,
    \\
    \phi^\ast q_2   &=  z_3 w_3 = p_2,
    &
    \phi^\ast q_3   &=  - z_3 w_2^2 = - p_4,
    \\
    \phi^\ast q_4   &=  - z_2^2 w_3 = - p_3,
    &
    \phi^\ast q_5   &=  - \sqrt{-1} w_1^3 w_3 = - \sqrt{-1} p_6,
    \\
    \phi^\ast q_6   &=  - \sqrt{-1} z_1^3 z_3 = - \sqrt{-1} p_5,
    &
    \phi^\ast q_7   &=  \sqrt{-1} w_1^3 w_2^2 = \sqrt{-1} p_8,
    \\
    \phi^\ast q_8   &=  \sqrt{-1} z_1^3 z_2^2 = \sqrt{-1} p_7,
\end{align*}
so that $\phi^\ast J_B = J_A$. Hence $\phi^\ast$ induces an isomorphism
$\R[M_0(B)]\otimes_{\R}\C\to\R[M_0(A)]\otimes_{\R}\C$, completing the proof.
\end{proof}

Clearly, the isomorphism $\phi^\ast$ does not restrict to a map
$\R[M_0(B)]\to\R[M_0(A)]$ of the real algebras. Hence, to determine an isomorphism over
$\R$, we need a more explicit description of
$\R[M_0(A)]$ and $\R[M_0(B)]$.

\begin{proof}[Proof of Proposition~\ref{prop:HilbCounter}({\it ii.})]
Using \emph{Macaulay2} \cite{M2}, we compute the relations among the generators
$p_1, p_2,\ldots,p_8$ of $\R[M_0(A)]$ to be
\begin{gather*}
    2 p_0-3 p_1-6 p_2,\quad\quad
    p_1^2 p_2-p_4 p_3,\quad\quad
    p_4 p_6-p_2 p_8,\quad\quad
    p_3 p_5-p_2 p_7,
    \\
    p_1^2 p_6-p_3 p_8,\quad\quad
    p_1^2 p_5-p_4 p_7,\quad\quad
    324 p_1 p_2^3+216 p_2^4+27 p_1 p_4 p_3+162 p_2 p_4 p_3-8 p_5 p_6,
    \\
    27 p_1^3 p_3+324 p_1 p_2^2 p_3+216 p_2^3 p_3+162 p_4 p_3^2-8 p_6 p_7,
    \\
    27 p_1^3 p_4+324 p_1 p_2^2 p_4+216 p_2^3 p_4+162 p_4^2 p_3-8 p_5 p_8,
    \\
    432 p_2^5-81 p_1^2 p_4 p_3-432 p_1 p_2 p_4 p_3-648 p_2^2 p_4 p_3+24 p_1 p_5 p_6-16 p_2 p_5 p_6,
    \\
    27 p_1^5+162 p_1^2 p_4 p_3+324 p_1 p_2 p_4 p_3+216 p_2^2 p_4 p_3-8 p_7 p_8,
    \\
    324 p_1 p_2^2 p_3 p_6+216 p_2^3 p_3 p_6-8 p_6^2 p_7+27 p_1 p_3^2 p_8+162 p_2 p_3^2 p_8,
    \\
    324 p_1 p_2^2 p_4 p_5+216 p_2^3 p_4 p_5+27 p_1 p_4^2 p_7+162 p_2 p_4^2 p_7-8 p_5^2 p_8,
    \\
    432 p_2^4 p_3 p_6+24 p_1 p_6^2 p_7-16 p_2 p_6^2 p_7-81 p_1^2 p_3^2 p_8-432 p_1 p_2 p_3^2 p_8-648 p_2^2 p_3^2 p_8,
    \\
    432 p_2^4 p_4 p_5-81 p_1^2 p_4^2 p_7-432 p_1 p_2 p_4^2 p_7-648 p_2^2 p_4^2 p_7+24 p_1 p_5^2 p_8-16 p_2 p_5^2 p_8.
\end{gather*}
Similarly, the relations among the generators $q_1,q_2,\ldots,q_8$ of $\R[M_0(B)]$ are given by
\begin{gather*}
    3 q_0-2 q_1-6 q_2,\quad\quad
    4 q_1^2 q_2+24 q_1 q_2^2+36 q_2^3-9 q_3 q_4,\quad\quad
    q_4 q_6-q_2 q_8,\quad\quad
    q_3 q_5-q_2 q_7,
    \\
    4 q_1^2 q_6+24 q_1 q_2 q_6+36 q_2^2 q_6-9 q_3 q_8,\quad\quad
    4 q_1^2 q_5+24 q_1 q_2 q_5+36 q_2^2 q_5-9 q_4 q_7,
    \\
    108 q_1 q_2^3+216 q_2^4+9 q_1 q_3 q_4-54 q_2 q_3 q_4-4 q_5 q_6,\quad\quad
    q_1^3 q_4-q_5 q_8,\quad\quad
    q_1^3 q_3-q_6 q_7,
    \\
    108 q_2^5-9 q_1^2 q_3 q_4+18 q_1 q_2 q_3 q_4-27 q_2^2 q_3 q_4+4 q_1 q_5 q_6+16 q_2 q_5 q_6,
    \\
    4 q_1^5+24 q_1 q_5 q_6+36 q_2 q_5 q_6-9 q_7 q_8,
    \\
    108 q_1 q_2^2 q_3 q_6+216 q_2^3 q_3 q_6-4 q_6^2 q_7+9 q_1 q_3^2 q_8-54 q_2 q_3^2 q_8,
    \\
    108 q_1 q_2^2 q_4 q_5+216 q_2^3 q_4 q_5+9 q_1 q_4^2 q_7-54 q_2 q_4^2 q_7-4 q_5^2 q_8,
    \\
    108 q_2^4 q_3 q_6+4 q_1 q_6^2 q_7+16 q_2 q_6^2 q_7-9 q_1^2 q_3^2 q_8+18 q_1 q_2 q_3^2 q_8-27 q_2^2 q_3^2 q_8,
    \\
    108 q_2^4 q_4 q_5-9 q_1^2 q_4^2 q_7+18 q_1 q_2 q_4^2 q_7-27 q_2^2 q_4^2 q_7+4 q_1 q_5^2 q_8+16 q_2 q_5^2 q_8.
\end{gather*}
Define the map $\Psi\co\R[M_0(A)]\to\R[M_0(B)]$ by
\begin{align*}
    \Psi (p_1)   &=  q_1 + 3 q_2,
    &
    \Psi (p_2)   &=  -\frac{3}{2} q_2,
    &
    \Psi (p_3)   &=  q_4,
    \\
    \Psi (p_4)   &=  -\frac{27}{8} q_3,
    &
    \Psi (p_5)   &=  q_6,
    &
    \Psi (p_6)   &=  -\frac{81}{16} q_5,
    \\
    \Psi (p_7)   &=  -\frac{2}{3} q_8,
    &
    \Psi (p_8)   &=  -\frac{729}{64} q_7.
\end{align*}
A tedious though elementary computation demonstrates that $\Psi$ maps the ideal of relations
of the $p_i$ into the ideal of relations of the $q_i$, and $\Psi^{-1}$ similarly maps the ideal
of relations of the $q_i$ into the ideal of relations of the $p_j$. Therefore,
$\Psi\co\R[M_0(A)]\to\R[M_0(B)]$ is an isomorphism. Note that $p_2 = z_3\overline{z_3} \geq 0$,
while $\Psi(p_2) = -3 q_2/2 \leq 0$ so that $\Psi$ does not preserve the inequalities.

To show that any graded isomorphism $\R[M_0(A)]\to\R[M_0(B)]$ fails to preserve the inequalities,
suppose for contradiction that $\Phi\co\R[M_0(A)]\to\R[M_0(B)]$ is such a graded isomorphism.
Let $\mathcal{Q}(A)$ and $\mathcal{Q}(B)$ denote the subalgebras of $\R[M_0(A)]$ and $\R[M_0(B)]$,
respectively, that are generated by elements of degree at most four. Then $\Phi$ restricts to an
isomorphism $\Phi|_{\mathcal{Q}(A)}\co\mathcal{Q}(A)\to\mathcal{Q}(B)$.

Again using \emph{Macaulay2} \cite{M2}, the algebra $\mathcal{Q}(A)$ generated by $p_1, p_2, \ldots, p_6$
has relations generated by
\begin{gather*}
    R_1 = p_1^2 p_2 - p_3 p_4, \quad\quad
    R_2 = 27 \big(4 p_2^3 (3 p_1+2 p_2)+(p_1+6 p_2) p_3 p_4\big) - 8 p_5 p_6,
    \\
    R_3 = -81 p_1^2 p_3 p_4-8 (-54 p_2^5+54 p_1 p_2 p_3 p_4+81 p_2^2 p_3 p_4-3 p_1 p_5 p_6+2 p_2 p_5 p_6),
    \\
    R_4 = 27 p_3 p_4 (p_1^3+12 p_1 p_2^2+8 p_2^3+6 p_3 p_4) - 8 p_1^2 p_5 p_6,
\end{gather*}
and the algebra $\mathcal{Q}(B)$ generated by $q_1, q_2, \ldots, q_6$ has relations generated by
\begin{gather*}
    R_1^\prime = 4 q_2 (q_1 + 3 q_2)^2 - 9 q_3 q_4, \quad\quad
    R_2^\prime = 108 q_2^3 (q_1 + 2 q_2) + 9 (q_1 - 6 q_2) q_3 q_4 - 4 q_5 q_6,
    \\
    R_3^\prime = 108 q_2^5 - 9 (q_1^2 - 2 q_1 q_2 + 3 q_2^2) q_3 q_4 + 4 (q_1 + 4 q_2) q_5 q_6,
    \\
    R_4^\prime = 9 q_1^3 q_3 q_4 - 4 (q_1 + 3 q_2)^2 q_5 q_6.
\end{gather*}
As $\Phi$ preserves the grading, it must be of the form
\begin{align}
    \nonumber
    \Phi(p_1)
        &=      c_{11} q_1 + c_{12} q_2,
    &
    \Phi(p_2)
        &=      c_{21} q_1 + c_{22} q_2,
    &
    \Phi(p_3)
        &=      c_{33} q_3 + c_{34} q_4,
    \\ \label{eq:DefPhi}
    \Phi(p_4)
        &=      c_{43} q_3 + c_{44} q_4,
    &
    \Phi(p_5)
        &=      c_{55} q_5 + c_{56} q_6,
    &
    \Phi(p_6)
        &=      c_{65} q_5 + c_{66} q_6,
    \\ \nonumber
    \Phi(p_7)
        &=      c_{77} q_7 + c_{78} q_8,
    &
    \Phi(p_8)
        &=      c_{87} q_7 + c_{88} q_8.
\end{align}
Using the fact that $\Phi$ preserves the grading and maps the ideal of relations for
the $p_i$ into the ideal of relations for the $q_i$, we must have
\[
    \Phi(R_1)   =   k_1 R_1^\prime,
    \quad\quad\mbox{and}\quad\quad
    \Phi(R_2)   =   k_2 R_2^\prime + k_3 q_1 R_1^\prime + k_4 q_2 R_2^\prime
\]
for some $k_1, k_2, k_3, k_4 \in \R$.
Computing the $q_1^3$, $q_3^2$, and $q_1^2 q_2$ coefficients of each side of the first equation
and the $q_1^2 q_2^2$, $q_2^4$, $q_1 q_2^3$, $q_2 q_3 q_4$, and $q_1 q_3 q_4$ coefficients of each side
of the second equation yields the system
\begin{align*}
    \hline
    \Phi(R_1):
    &&q_1^3:
    &&&c_{11}^2 c_{21} =      0,
    \\
    &&q_3^2:
    &&&c_{33} c_{43}   =      0,
    \\
    &&q_1^2 q_2:
    &&&c_{11} (2 c_{12} c_{21} + c_{11} c_{22})  =      4 k_1,
    \\
    \hline
    \Phi(R_2):
    &&q_1^2 q_2^2:
    &&&81 c_{21} c_{22} (3 c_{12} c_{21} + 3 c_{11} c_{22} + 4 c_{21} c_{22})
            =      k_2 (6 k_3 + k_4),
    \\
    &&q_1 q_2^3:
    &&&9 c_{22}^2 (9 c_{12} c_{21} + 3 c_{11} c_{22} + 8 c_{21} c_{22})
            =      k_2 (9 + 3 k_3 + 2 k_4),
    \\
    &&q_2^4:
    &&&3 c_{22}^3 (3 c_{12} + 2 c_{22})          =      k_2 (6 + k_4),
    \\
    &&q_1 q_3 q_4:
    &&&3(c_{11} + 6 c_{21}) (c_{34} c_{43} + c_{33} c_{44})
            =       k_2 (1 - k_3),
    \\
    &&q_2 q_3 q_4:
    &&&3 (c_{12} + 6 c_{22}) (c_{34} c_{43} + c_{33} c_{44})
            =      -k_2 (6 + k_4),
    \\ \hline
\end{align*}
Every solution of this system not corresponding to $\Phi(p_i) = 0$ for some $i$
satisfies $c_{11} = - 2 c_{22}/3$, $c_{12} = - 2 c_{22}$, and $c_{21} = 0$.
Hence, though $p_1 \geq 0$ and $p_2 \geq 0$, either $c_{22} > 0$ so that
$\Phi(p_1) = - 2 c_{22} (q_1/3 + q_2) < 0$ for any nonzero $q_1$ or $q_2$,
or $c_{22} < 0$ so that $\Phi(p_2) = c_{22} q_2 < 0$ for any nonzero $q_2$.
In either case, $\Phi$ does not preserve the inequalities describing the
semilagebraic sets $M_0(A)$ and $M_0(B)$.
\end{proof}

As another example, let $A^\prime = (-2,1,1)$, Type II$_2$
with $\alpha = 2$, $\beta = 1$, $c_1 = 1$, and $c_2 = 1$; and let
$B^\prime = (-1,2,1)$, Type II$_2$ with $\alpha = 1$, $\beta = 1$, $c_1 = 2$, and $c_2 = 1$.
As above, $\R[M_0(A^\prime)]$ and $\R[M_0(B^\prime)]$ have the same Hilbert series,
given by
\[
    \frac{1+2 t^2+4 t^3+2 t^4+t^6}{(1-t^3)^2 (1-t^2)^2}.
\]
The quadratic off-shell invariants of the action with weight matrix $A^\prime$ are spanned by
$z_1 \overline{z_1}$, $z_2 \overline{z_2}$, $z_3 \overline{z_3}$, $z_2 \overline{z_3}$, and
$z_3 \overline{z_2}$ with relation
$(z_2\overline{z_2}) (z_3\overline{z_3}) = (z_2\overline{z_3}) (z_3\overline{z_2})$,
and the moment map determines $z_1 \overline{z_1}$ in terms of
$z_2 \overline{z_2}$, $z_3 \overline{z_3}$.
For the action with weight matrix $B^\prime$, the quadratic off-shell invariants are generated by
$u_1 \overline{u_1}$, $u_2 \overline{u_2}$, $u_3 \overline{u_3}$, $u_1 u_3$, and $\overline{u_1} \overline{u_3}$ with relation
$(u_1 u_3)(\overline{u_1}\overline{u_3}) = (u_1 \overline{u_1}) (u_3 \overline{u_3})$,
and the moment map expresses $u_1 \overline{u_1} = 2u_2 \overline{u_2} + u_3 \overline{u_3}$. Considering only the Poisson brackets of the quadratics,
computations similar to those above demonstrate that
any graded Poisson isomorphism $\Phi\co\R[M_0(B^\prime)]\to\R[M_0(A^\prime)]$
must map
$u_3\overline{u_3}\mapsto c z_2 \overline{z_2} + (c-1) z_3\overline{z_3}
+ \sqrt{-1} d z_3 \overline{z_2}$ where $c \in \{0,1\}$ and $d\neq 0$.
For each $z_2,z_3\in\C$, there is a $z_1\in\C$ such that $(z_1,z_2,z_3)\in Z_{A^\prime}$
so that $z_3\overline{z_2}$ is not bounded by inequalities.
As $u_3\overline{u_3}, z_2\overline{z_2}, z_3\overline{z_3}\geq 0$,
it follows that $\Phi$ cannot preserve the inequalities.

Finally, we consider a closely related example that is not of Type I$_k$ nor II$_k$
for any $k$. Let
\[
    A^{\prime\prime}
    = \begin{pmatrix}
        -1 & 0 & 1 & 1 \\
        0 & -1 & 1 & 1
    \end{pmatrix}
    \quad\quad\mbox{and}\quad\quad
    B^{\prime\prime} =
    \begin{pmatrix}
        -1 & 0 & 1 & 1\\
        0 & -1 & 0 & 1
    \end{pmatrix}.
\]
To see that the Hilbert series of $\R[M_0(A^{\prime\prime})]$ and $\R[M_0(B^{\prime\prime})]$
coincide note that the cotangent-lifted weight matrix corresponding to $A^{\prime\prime}$,
\[
    \begin{pmatrix}
        -1 & 0 & 1 & 1 &|& 1 & 0 & -1 & -1 \\
        0 & -1 & 1 & 1 &|& 0 & 1 & -1 & -1
    \end{pmatrix},
\]
can be transformed into that of $B^{\prime\prime}$,
\[
    \begin{pmatrix}
        -1 & 0 & 1 & 1 &|& 1 & 0 & -1 & -1 \\
        0 & -1 & 0 & 1 &|& 0 & 1 &  0 & -1
    \end{pmatrix}
\]
by transposing the column pairs $(1,4)$, $(3,7)$, $(5,8)$ and row-reducing over $\Z$.
The common Hilbert series is given by
\[
    \frac{1 + 2t^2 + 2t^3 + 2t^4 + t^6}{(1 - t^3)^2(1 - t^2)^2}.
\]
The quadratic off-shell invariants associated to $A^{\prime\prime}$ are
$z_1\overline{z_1}$, $z_2\overline{z_2}$, $z_3\overline{z_3}$,
$z_4\overline{z_4}$, $z_3\overline{z_4}$, and $z_4\overline{z_3}$,
the moment map expresses $z_1\overline{z_1}$ and $z_2\overline{z_2}$
in terms of $z_3\overline{z_3}$ and $z_4\overline{z_4}$,
and we have the relation
$(z_3\overline{z_4})(z_4\overline{z_3}) = (z_3\overline{z_3})(z_4\overline{z_4})$.
Similarly, the quadratic off-shell invariants associated to $B^{\prime\prime}$ are
$z_1\overline{z_1}$, $z_2\overline{z_2}$, $z_3\overline{z_3}$,
$z_4\overline{z_4}$, $z_1 z_3$, and $\overline{z_1}\overline{z_3}$,
the moment map expresses $z_1\overline{z_1}$ and $z_2\overline{z_2}$
in terms of $z_3\overline{z_3}$ and $z_4\overline{z_4}$, and we have the relation
$(z_1 z_3)(\overline{z_1}\overline{z_3}) = (z_2\overline{z_2} + z_3\overline{z_3})(z_3\overline{z_3})$.
Hence, computations identical to those for $A^\prime$ and $B^\prime$ demonstrate that
the only Poisson isomorphisms
between the algebras $\R[M_0(A^{\prime\prime})]$ and $\R[M_0(B^{\prime\prime})]$
do not satisfy the semialgebraic conditions,
and hence do not correspond to a graded regular symplectomorphism.


\bibliographystyle{amsplain}
\bibliography{HLS-Torus}

\providecommand{\bysame}{\leavevmode\hbox to3em{\hrulefill}\thinspace}
\providecommand{\MR}{\relax\ifhmode\unskip\space\fi MR }
\providecommand{\MRhref}[2]{%
  \href{http://www.ams.org/mathscinet-getitem?mr=#1}{#2}
}
\providecommand{\href}[2]{#2}
\begin{thebibliography}{10}

\bibitem{ArmsGotayJennings}
Judith~M. Arms, Mark~J. Gotay, and George Jennings, \emph{Geometric and
  algebraic reduction for singular momentum maps}, Adv. Math. \textbf{79}
  (1990), no.~1, 43--103.

\bibitem{FarHerSea}
Carla Farsi, Hans-Christian Herbig, and Christopher Seaton, \emph{On orbifold
  criteria for symplectic toric quotients}, SIGMA Symmetry Integrability Geom.
  Methods Appl. \textbf{9} (2013), Paper 032, 33. \MR{3056176}

\bibitem{M2}
Daniel~R. Grayson and Michael~E. Stillman, \emph{Macaulay2, a software system
  for research in algebraic geometry}, Available at
  \texttt{http://www.math.uiuc.edu/Macaulay2/}, 2012.

\bibitem{HerbigIyengarPflaum}
Hans-Christian Herbig, Srikanth~B. Iyengar, and Markus~J. Pflaum, \emph{On the
  existence of star products on quotient spaces of linear {H}amiltonian torus
  actions}, Lett. Math. Phys. \textbf{89} (2009), no.~2, 101--113. \MR{2534878}

\bibitem{HerbigSchwarz}
Hans-Christian Herbig and Gerald~W. Schwarz, \emph{The {K}oszul complex of a
  moment map}, J. Symplectic Geom. \textbf{11} (2013), no.~3, 497--508.

\bibitem{HerbigSchwarzSeaton}
Hans-Christian Herbig, Gerald~W. Schwarz, and Christopher Seaton, \emph{When is
  a symplectic quotient an orbifold?}, Adv. Math. \textbf{280} (2015),
  208--224.

\bibitem{HerbigSchwarzSeaton2}
\bysame, \emph{Symplectic quotients have symplectic singularities}, preprint,
  submitted (2017), arXiv:1706.02089 [math.SG].

\bibitem{HerbigSeatonHilbSympCirc}
Hans-Christian Herbig and Christopher Seaton, \emph{The {H}ilbert series of a
  linear symplectic circle quotient}, Exp. Math. \textbf{23} (2014), no.~1,
  46--65.

\bibitem{HerbigSeaton2}
\bysame, \emph{An impossibility theorem for linear symplectic circle
  quotients}, Rep. Math. Phys. \textbf{75} (2015), no.~3, 303--331.
  \MR{3352005}

\bibitem{PopovSeshadri}
V.~L. Popov, \emph{On the ``lemma of {S}eshadri''}, Lie groups, their discrete
  subgroups, and invariant theory, Adv. Soviet Math., vol.~8, Amer. Math. Soc.,
  Providence, RI, 1992, pp.~167--172. \MR{1155673}

\bibitem{PopovVinberg}
V.~L. Popov and {\`E}.~B. Vinberg, \emph{Invariant theory}, Algebraic geometry.
  {IV}, Encyclopaedia of Mathematical Sciences, vol.~55, Springer-Verlag,
  Berlin, 1994, Linear algebraic groups. Invariant theory, A translation of
  {{\i}t Algebraic geometry. 4} (Russian), Akad. Nauk SSSR Vsesoyuz. Inst.
  Nauchn. i Tekhn. Inform., Moscow, 1989 [ MR1100483 (91k:14001)], Translation
  edited by A. N. Parshin and I. R. Shafarevich, pp.~vi+284.

\bibitem{ProcesiSchwarz}
Claudio Procesi and Gerald Schwarz, \emph{Inequalities defining orbit spaces},
  Invent. Math. \textbf{81} (1985), no.~3, 539--554.

\bibitem{SmoothInv}
Gerald~W. Schwarz, \emph{Smooth functions invariant under the action of a
  compact {L}ie group}, Topology \textbf{14} (1975), 63--68. \MR{0370643}

\bibitem{GWSLiftingHomotopies}
\bysame, \emph{Lifting smooth homotopies of orbit spaces}, Inst. Hautes
  \'Etudes Sci. Publ. Math. (1980), no.~51, 37--135. \MR{573821}

\bibitem{SjamaarLerman}
Reyer Sjamaar and Eugene Lerman, \emph{Stratified symplectic spaces and
  reduction}, Ann. of Math. (2) \textbf{134} (1991), no.~2, 375--422.

\bibitem{WattsSympQuotCircle}
Jordan Watts, \emph{Symplectic quotients and representability: the circle
  action case},  (2016), arXiv:1610.01547 [math.SG].

\end{thebibliography}

\end{document}